\documentclass[11pt,a4paper]{article}
\usepackage{amsthm}	
\usepackage{amsfonts, amssymb, amsmath, amscd, latexsym}
\usepackage{eucal, enumerate, latexsym}
\usepackage{graphics}
\usepackage{graphicx}

\newcommand{\R}{\mathbb R}

\newtheorem{theorem}{Theorem}[section]

\newtheorem{definition}{Definition}[section]
\newtheorem{remark}{Remark}[section]
\newtheorem{lemma}{Lemma}[section]
\newtheorem{example}{Example}[section]

\newcommand{\orho}{\overline{\rho}}
\newcommand{\ov}{\overline{v}}
\newcommand{\oxi}{\overline{\xi}}
\newcommand{\ds}{\displaystyle}

\renewcommand{\epsilon}{\varepsilon}
\newcommand{\eps}{\epsilon}
\renewcommand{\phi}{\varphi}
\newenvironment{proofof}[1]{\smallskip\noindent\emph{Proof of #1.}%
\hspace{1pt}}{\hspace{-5pt}{\nobreak\quad\nobreak\hfill\nobreak%
$\square$\vspace{8pt}\par}\smallskip\goodbreak}

\makeatletter

\newlength{\captionwidth}
\long\def\@makecaption#1#2{%
   \vskip 10\p@
   \setbox\@tempboxa\hbox{#1: #2}%
   \ifdim \wd\@tempboxa > \captionwidth 
       \hbox to\hsize{\hfil
       \parbox[t]{\captionwidth}{
       \small#1: \small#2\par}
       \hfil}
     \else
       \hbox to\hsize{\hfil\box\@tempboxa\hfil}%
   \fi}

\makeatother

\begin{document}

\title{Viscous profiles in models of collective movements\\ with negative diffusivities}

\author{Andrea Corli\\
\small\textit{Department of Mathematics and Computer Science, University of Ferrara}\\
\small\textit{I-44121 Italy, e-mail: andrea.corli@unife.it},
\and Luisa Malaguti\\
\small\textit{Department of Sciences and Methods for Engineering, University of Modena and Reggio Emilia}\\
\small\textit{I-42122 Italy, e-mail: luisa.malaguti@unimore.it}}

\maketitle
\begin{abstract}
In this paper we consider an advection-diffusion equation, in one space dimension, whose diffusivity can be negative. Such equations arise in particular in the modeling of vehicular traffic flows or crowds dynamics, where a negative diffusivity simulates aggregation phenomena. We focus on traveling-wave solutions that connect two states whose diffusivity has different signs; under some geometric conditions we prove the existence, uniqueness (in a suitable class of solutions avoiding plateaus) and sharpness of the corresponding profiles. Such results are then extended to the case of end states where the diffusivity is positive but it becomes negative in some interval between them. Also the vanishing-viscosity limit is considered. At last, we provide and discuss several examples of diffusivities that change sign and show that our conditions are satisfied for a large class of them in correspondence of real data.

\vspace{1cm}

\end{abstract}

\section{Introduction}\label{s:I}

We are interested in the advection-diffusion equation
\begin{equation}\label{e:E}
\rho_t + f(\rho)_x=\left(D(\rho)\rho_x\right)_x, \qquad t\ge 0, \, x\in \R.
\end{equation}
The unknown variable $\rho$ is understood as a density or concentration and is valued in the interval $[0,1]$; the flux $f$ and the diffusivity $D$ are smooth functions. The main assumption of our study is that $D$ {\em changes sign}, even more than once. Equation \eqref{e:E} is then a forward-backward parabolic equation. In spite of the fact that these equations are well-known to be unstable in the backward regime, nevertheless they arise in a natural way in several physical and biological models, see \cite{Kerner-Osipov, Padron} and references there.

\smallskip

Equation \eqref{e:E} also arises in the modeling of {\em collective movements}, for instance vehicular traffic flows or crowds dynamics. In these cases, vehicles or pedestrians are assumed to move along a straight road or corridor, respectively; their normalized density at time $t$ and place $x$ is represented by $\rho(x,t)$. The corresponding flow is $f\left(\rho\right)=\rho v(\rho)$, where the velocity $v$ is an assigned function.

The first models proposed in \cite{Lighthill-Whitham, Richards} had no diffusivity; we refer to \cite{Garavello-Han-Piccoli_book, Rosinibook} for updated information on this case. Then, the density-flow pairs lie on a curve (the graph of $f$) in the $(\rho,f)$-plane. However, experimental data show that this is not the case \cite{Helbing, Kerner}: such pairs usually cover a two-dimensional region. To reproduce this effect, either one considers second-order models \cite{Aw-Rascle, Payne, Zhang} or, as in this paper, introduces a diffusive term. In the latter case the physical flow is $q=f(\rho) - D(\rho)\rho_x$, see \cite{Bellomo-Delitala-Coscia, Bellomo-Dogbe, BTTV, Nelson_2000}, and the density-flow pairs now correctly cover a full two-dimensional region in the $(\rho,q)$-plane. Moreover, the introduction of $D$ avoids the appearance of shock waves and then the occurrence of an infinite acceleration, which do not seem to fit well with the usual perception of collective flows.

We refer to \cite{Bellomo-Delitala-Coscia, Bellomo-Dogbe, BTTV} for several models where the diffusivity $D$ can vanish but otherwise remains positive. The {\em negativity} of $D$ simulates an aggregative behavior; it occurs, for instance, in vehicular flows for high car densities and limited sight distance ahead \cite{Nelson_2000}. An analogous modeling can be made in the framework of crowds dynamics; in this case, as we propose in this paper, it may simulates panic behaviors in overcrowded environments \cite{{Colombo-Rosini2005}}.

\smallskip

In this paper we are concerned with solutions $\rho$  of equation \eqref{e:E} having a low density in the past and a higher density in the future, or conversely. To this aim, we restrict our investigation to {\em traveling-wave solutions} $\rho(x,t)=\phi(x-ct)$. The  profile $\phi$ satisfies the differential equation
\begin{equation}\label{e:ODE}
\left( D(\phi)\phi^{\prime}\right)^{\prime}+(c\phi -f(\phi))^{\prime}=0,
\end{equation}
where the constant $c$ denotes its speed; more precisely, solutions are meant in the weak sense according to the following Definition \ref{d:tws}. We are interested, in particular, in wavefront solutions connecting a value $\ell^-$ to a value $\ell^+$, i.e., such that
\begin{equation}\label{e:infty}
\phi(-\infty)=\ell^-, \qquad \phi(+\infty)=\ell^+,
\end{equation}
either with $D(\ell^-)>0$ and $D(\ell^+)<0$, or with $D(\ell^\pm)>0$ but then $D$ is negative in an interval contained in $(\ell^-,\ell^+)$. As far as wavefronts are concerned, the special case when $f=0$ but equation \eqref{e:E} is endowed of a source term $g$ has been considered by many authors; we refer to \cite{Bao-Zhou2014, Bao-Zhou2017, Maini-Malaguti-Marcelli-Matucci2006} for $D$ changing sign once and monostable $g$, \cite{Maini-Malaguti-Marcelli-Matucci2007} for the bistable case, \cite{Ferracuti-Marcelli-Papalini, Kuzmin-Ruggerini} for $D$ changing sign twice where $g$ is, respectively, monostable and bistable.

We refer to \cite{Nelson_2000, Nelson_2002} for several interpretations of the wavefront solutions to equation \eqref{e:E} in the case $D(\rho)= -\rho v'(\rho)\left(\delta+\tau\rho v'(\rho)\right)$ and in the framework of vehicular flows. Here, $\delta$ is an anticipation distance and $\tau$ a reaction time. However, even if also the case when $D<0$ is commented in \cite{Nelson_2000}, in \cite{Nelson_2002} only the case when $D$ is positive is treated.

\smallskip

About equation \eqref{e:E}, the existence of wavefront solutions in intervals where $D$ does not change sign follows by a result in \cite{GK}; this topic is briefly discussed in Section \ref{s:preliminary}, where we extend that result to cover the case when $D$ is negative. On the other hand, the case when a wavefront solution crosses an interval where $D$ changes sign is more delicate and, to the best of our knowledge, has never been considered. Our main results are provided in Section \ref{sec:main} and give necessary and sufficient conditions for the existence of wavefronts in the cases $D$ changes sign once or twice; generalizations to further or opposite changes of sign are straightforward. We also study the smoothness of the profiles, in particular at the singular points where $D$ changes sign. At last, we prove the vanishing-viscosity limit of the wavefronts to discontinuous (nonentropic) solutions of the corresponding hyperbolic conservation law
\begin{equation}\label{e:cl}
\rho_t+f(\rho)_x=0.
\end{equation}
As a consequence, Theorem \ref{t:main2} below shows that some nonclassical shock waves \cite{LeFloch} considered in \cite{Colombo-Rosini2005} (in the hyperbolic regime $D=0$), in the modeling of panic situation in crowds dynamics, admit a viscous profile. This topic is also discussed in Section \ref{sec:main}. The proofs of our results are given in Section \ref{s:proof}. The main applications are collected in Section \ref{sec:ex}; further examples are provided in Section \ref{s:preliminary} and \ref{s:proof}.

\section{Main results}\label{sec:main}
\setcounter{equation}{0}
We assume that the flux function $f$ and the diffusivity $D$ satisfy, for some $\alpha\in(0,1)$,
\begin{itemize}
\item[{(f)}]\, $f\in C^1[0,1]$, $f(0)=0$;

\item[{(D1)}] \, $D\in C^1[0,1]$, $D(\rho)>0$ for $\rho \in (0, \alpha)$ and $D(\rho)<0$ for $\rho\in(\alpha, 1)$.

\end{itemize}
\noindent
We notice that in this case we clearly have $D(\alpha)=0$; moreover, the condition $f(0)=0$ in (f) is not really an assumption because $f$ is defined up to an additive constant. We warn the reader that in Figure \ref{f:f} and in the following ones we represent $f$ with $f(\rho)>0$ in $(0,1)$, $f(1)=0$ and $D$ with $D(0)=D(1)=0$. Such assumptions are common in dealing with collective movements, see Section \ref{sec:ex}, but {\em are by no means necessary} for the results below.

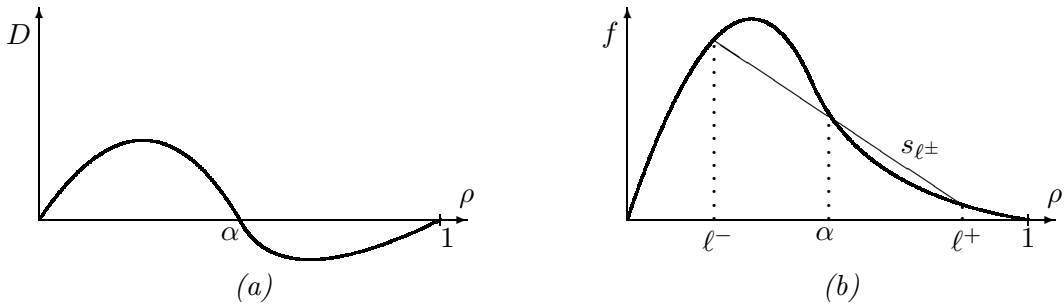
\begin{figure}[htbp]
\begin{picture}(100,125)(80,-25)
\setlength{\unitlength}{1pt}

\put(110,0){
\put(0,0){\vector(1,0){160}}
\put(160,8){\makebox(0,0){$\rho$}}
\put(0,0){\vector(0,1){80}}
\put(-3,70){\makebox(0,0)[r]{$D$}}

\put(150,-3){\makebox(0,0)[tl]{$1$}}
\put(150,-2){\line(0,1){4}}

\put(0,0){\thicklines{\qbezier(0,0)(40,60)(75,0)}}
\put(0,0){\thicklines{\qbezier(75,0)(90,-30)(150,0)}}

\put(75,-3){\makebox(0,0)[tr]{$\alpha$}}

\put(80,-20){\makebox(0,0)[t]{{\em (a)}}}
}

\put(330,0){
\put(0,0){\vector(1,0){160}}
\put(160,8){\makebox(0,0){$\rho$}}
\put(0,0){\vector(0,1){80}}
\put(-3,70){\makebox(0,0)[r]{$f$}}

\put(150,-3){\makebox(0,0)[t]{$1$}}
\put(150,-2){\line(0,1){4}}

\put(0,0){\thicklines{\qbezier(0,0)(40,120)(70,50)}}
\put(0,0){\thicklines{\qbezier(70,50)(88,10)(150,0)}}
\put(65,46){\line(3,-2){60}}
\put(65,46){\line(-3,2){33}}

\multiput(74,0)(0,5){8}{$.$}
\put(74,-3){\makebox(0,0)[t]{$\alpha$}}

\put(110,30){\makebox(0,0)[t]{$s_{\ell^\pm}$}}

\multiput(124,0)(0,4){2}{$.$}
\put(127,-3){\makebox(0,0)[t]{$\ell^+$}}

\multiput(31,0)(0,5){14}{$.$}
\put(34,-3){\makebox(0,0)[t]{$\ell^-$}}
\put(80,-20){\makebox(0,0)[t]{{\em (b)}}}
}

\end{picture}
\caption{\label{f:f}{{\em (a)}: a diffusivity $D$ satisfying assumption (D1); {\em (b)}: the flux function $f$. }}
\end{figure}

Now, we recall some definitions on traveling-wave solutions \cite{GK}; they are given under assumptions somewhat weaker than (f) and (D1).

\begin{definition}\label{d:tws} Let $I\subseteq \R$ be an open interval, $D\in C[0,1]$, $f\in C[0,1]\cap C^1(0,1)$, $f(0)=0$; consider a function $\varphi \colon I \to [0, 1]$ such that $\varphi\in C(I)$ and $D(\varphi) \varphi^{\, \prime}\in L_{\rm loc}^1(I)$. For all $(x,t)$ with $x-ct \in I$, the function $\rho(x,t)=\varphi(x-ct)$ is said a {\em traveling-wave} solution of equation \eqref{e:E} with wave speed $c$ and wave profile $\phi$ if
\begin{equation}\label{e:def-tw}
\int_I \left(D\left(\phi(\xi)\right)\phi'(\xi) - f\left(\phi(\xi)\right) + c\phi(\xi) \right)\psi'(\xi)\,d\xi =0,
\end{equation}
for every $\psi\in C_0^\infty(I)$.

A traveling-wave solution is {\em global} if $I=\R$ and {\em strict} if $I\ne \R$ and $\phi$ is not extendible to $\R$. It is {\em classical} if $\varphi$ is differentiable, $D(\varphi) \varphi'$ is absolutely continuous and \eqref{e:ODE} holds a.e.; at last, it is {\em sharp at $\ell$} if there exists $\xi_{\ell}\in I$ such that $\phi(\xi_{\ell})=\ell$, with $\phi$ classical in $I\setminus\{\xi_0\}$ and not differentiable at $\xi_{\ell}$.

A global, bounded traveling-wave solution with a monotonic, non-constant profile $\phi$  satisfying \eqref{e:infty} with $\ell^-, \ell^+ \in [0, 1]$ is said to be a {\em wavefront solution} from $\ell^-$ to  $\ell^+$.
\end{definition}

Above, monotonic is meant in the sense that $\xi_1<\xi_2$ implies $\phi(\xi_1)\le \phi(\xi_2)$. We denote by $s_{\ell^\pm}=s_{\ell^\pm}(\rho)$ the function whose graph is the line joining $\left(\ell^-, f(\ell^-)\right)$ with $\left(\ell^+, f(\ell^+)\right)$, see Figure \ref{f:f}{\em (b)}.

\begin{remark}\label{r:ac}
Assume {\rm (f)}, {\rm (D1)} and let $\rho$ be a traveling-wave solution of \eqref{e:E} with profile $\phi$ defined in $I$ and speed $c$. We claim that $\phi$ is classical in every interval $I_{\pm}\subseteq I$ where $D\left(\phi(\xi)\right)\gtrless 0$ for $\xi \in I_{\pm}$; indeed, $\phi\in C^2(I_\pm)$.

As far as the interval $I_+$ is concerned, this follows by \cite[Lem. 2.20 and Thm. 2.39]{GK}. About $I_-$, consider the equation
\begin{equation}\label{e:rtR}
r_t+g(r)_x=\left(E(r)r_x \right)_x, \quad r\in[0,1],
\end{equation}
with $E(r):=-D(r)$ and $g(r):=-f(r)$, $r \in [0,1]$.
It is immediate to prove that $\phi$ is a traveling-wave solution of \eqref{e:E} with speed $c$ if and only if it is a traveling-wave solution of \eqref{e:rtR} with opposite speed $-c$. Since $E$ and $D$ have opposite signs, we conclude that $\phi$ is classical also in $I_{-}$ and $\phi\in C^2(I_-)$.
\end{remark}

It is well-known that, for positive diffusivities, profiles are uniquely determined up to a shift: if $\phi(\xi)$ is a profile, then also $\phi(\xi+\bar\xi)$ is another profile for every $\bar\xi\in\R$ \cite{GK}. This still holds under (D1), but the loss of uniqueness is more severe as we now show. Assume that \eqref{e:E} admits a wavefront solution with profile $\phi$ and speed $c$. Since $\phi$ is monotone and continuous by Definition \ref{d:tws}, then there is a unique $\xi_0$, which we can assume to coincide with $0$ without loss of generality, such that
\begin{equation}\label{e:xi0}
\phi(0) =\alpha\quad \hbox{ and } \quad \phi(\xi)<\alpha\quad \hbox{ for } \quad\xi<0.
\end{equation}
Similarly, there is a unique $\xi_1=\xi_1\ge 0$ with
\begin{equation}\label{e:xi1}
\phi(\xi_1)=\alpha\quad  \hbox{ and } \quad \phi(\xi)>\alpha\quad \hbox{ for } \xi>\xi_1.
\end{equation}
If $\xi_1>0$, then $\phi\equiv \alpha$ in the whole interval $[0, \xi_1]$, see Figure \ref{f:TW}.

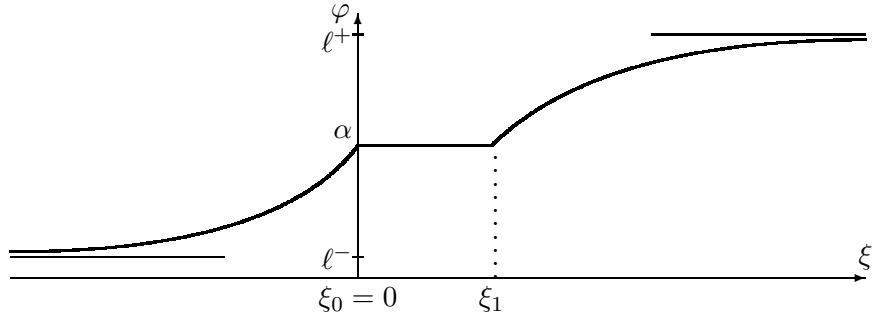
\begin{figure}[htbp]
\begin{picture}(100,125)(80,-15)
\setlength{\unitlength}{1pt}

\put(270,0){
\put(0,0){\vector(1,0){190}}
\put(0,0){\line(-1,0){130}}
\put(190,8){\makebox(0,0){$\xi$}}
\put(0,0){\vector(0,1){100}}
\put(-3,100){\makebox(0,0)[r]{$\phi$}}

\put(-2,8){\line(1,0){4}}
\put(-2,8){\makebox(0,0)[r]{$\ell^-$}}
\put(-130,8){\line(1,0){80}}

\put(-2,92){\line(1,0){4}}
\put(-2,90){\makebox(0,0)[r]{$\ell^+$}}
\put(190,92){\line(-1,0){80}}

\put(-2,53){\makebox(0,0)[rb]{$\alpha$}}

\put(0,0){\thicklines{\qbezier(-130,10)(-30,10)(0,50)}}
\put(0,0){\thicklines{\qbezier(0,50)(30,50)(50,50)}}
\put(0,0){\thicklines{\qbezier(50,50)(90,90)(190,90)}}

\put(0,-3){\makebox(0,0)[t]{$\xi_0=0$}}

\multiput(50,0)(0,5){10}{$.$}
\put(50,-3){\makebox(0,0)[t]{$\xi_1$}}

}

\end{picture}
\caption{\label{f:TW}{Under (D1), a profile $\phi$ in the case $\xi_1>0$.}}
\end{figure}

The following lemma shows that the case $\xi_1>0$ may really occur.

\begin{lemma}\label{l:lemma1}
Assume {\rm (f)} and {\rm (D)}; suppose that \eqref{e:E} admits a wavefront whose profile satisfies \eqref{e:infty}. Then, for every $\xi_1\ge0$ equation \eqref{e:E} has a wavefront whose profile still satisfies  \eqref{e:infty} and also \eqref{e:xi0}, \eqref{e:xi1}.
\end{lemma}

In other words, the loss of uniqueness of the profiles does not only concern shifts but also the \lq\lq stretching\rq\rq\ of the interval $[0,\xi_1]$, $\xi_1\ge0$. Notice that the latter loss of uniqueness does not occur if the diffusivity  possibly vanishes only at $0$ or $1$ \cite{GK}. Clearly, the profiles mentioned in Lemma \ref{l:lemma1} have the same smoothness.

For sake of simplicity, in the following we mainly focus on the case
\begin{equation}\label{e:xi0xi1=0}
\xi_0=\xi_1=0.
\end{equation}
The case $\xi_1>0$ (and related situations) will be discussed at the end of this section.

Now, we state the main results of this paper.

\begin{theorem}\label{t:main} \ Assume \emph{(f)} and \emph{(D1)}; let $\ell^- \in [0, \alpha)$ and $\ell^+ \in (\alpha, 1]$. Equation \eqref{e:E} admits a wavefront solution whose profile $\phi$ satisfies \eqref{e:infty}  if and only if
\begin{equation}\label{e:existence wfs}
\frac{f(\alpha)-f(\ell^-)}{\alpha-\ell^-}= \frac{f(\ell^+)-f(\alpha)}{\ell^+-\alpha}=:c_{\ell^{\pm}},
\end{equation}
and
\begin{equation}\label{e:existence about f}
\begin{array}{rll}
&f(\rho)>\ds\frac{f(\alpha)-f(\ell^-)}{\alpha-\ell^-}(\rho-\alpha)+f(\alpha), \quad &\rho \in (\ell^-, \alpha),
\\[3mm]
&f(\rho)<\ds\frac{f(\ell^+)-f(\alpha)}{\ell^+-\alpha}(\rho-\alpha)+f(\alpha), \quad &\rho \in (\alpha, \ell^+).
\end{array}
\end{equation}
The value $c_{\ell^{\pm}}$ is its wave speed and
\begin{equation}\label{e:slope}
f^{\, \prime}(\alpha) \le c_{\ell^{\pm}}.
\end{equation}
The profile $\phi$ is sharp in $\ell^-$ if and only if $\ell^-=0=D(0)$; $\phi$ is sharp in $\ell^+$ if and only if $1-\ell^+=0=D(1)$.

Assume \eqref{e:xi0xi1=0}. Then $\phi$ is unique; moreover, $\phi^{\, \prime}(\xi)>0$ when $\ell^-<\phi(\xi)<\ell^+$, $\xi\ne0$, while
\begin{equation}\label{e:slopeLR}
\lim_{\xi\to0}\phi'(\xi)=\left\{ \begin{array}{rl}
\frac{f^{\, \prime}(\alpha) - c_{\ell^{\pm}}}{D^{\, \prime}(\alpha) }& \text{ if } D^{\, \prime}(\alpha)  <0,
\\
\infty & \text{ if } D^{\, \prime}(\alpha)  =0 \hbox{ and }f^{\, \prime}(\alpha) - c_{\ell^{\pm}}<0.
\end{array}
\right.
\end{equation}
\end{theorem}
We refer to Figure \ref{f:f}{\em (b)} for the geometric meaning of conditions \eqref{e:existence wfs} and \eqref{e:existence about f}; in particular, $c_{\ell^{\pm}}=\left(f(\ell^+)-f(\ell^-)\right)/(\ell^+-\ell^-)$ is the slope of the line $s_{\ell^\pm}$. Notice that if condition \eqref{e:existence about f} is satisfied then $f$ has an inflection point, which however does not necessarily coincide with $\alpha$ (see for instance the case illustrated in Figure \ref{f:fN}). Equation \eqref{e:E} may have stationary wavefronts, i.e. with speed $c_{\ell^{\pm}}=0$; this happens  if and only if $f(\ell^-)=f(\ell^+)$ and then $f(\alpha)=f(\ell^\pm)$.

The case $f^{\, \prime}(\alpha) - c_{\ell^{\pm}}=D^{\, \prime}(\alpha)=0$ in \eqref{e:slopeLR} is a bit more delicate. Indeed, in the proof of Theorem \ref{t:main} we shall show
\begin{equation}\label{e:00}
\displaystyle{\lim_{\xi \to 0^-}}\phi^{\, \prime}(\xi)=\displaystyle{\lim_{\sigma \to \alpha}}
\frac{f(\sigma)-\left[f(\alpha) + f^{\, \prime}(\alpha)(\sigma-\alpha)\right]}{D(\sigma)}.
\end{equation}
Therefore, the existence and the value of that limit depends on the relative behavior of $f$ and $D$ in a neighborhood of $\alpha$.

The results of Theorem \ref{t:main} can be easily extended to the case when $D$ satisfies
\begin{itemize}
\item[{(D2)}] \, $D\in C^1[0,1]$, $D(\rho)>0$ for $\rho \in (0, \alpha)\cup (\beta, 1)$ and $D(\rho)<0$ for $\rho\in(\alpha, \beta)$.
\end{itemize}
We refer to Figure \ref{f:f2}{\em (a)} for a possible plot of a diffusivity $D$ satisfying (D2).

\begin{figure}[htbp]
\begin{picture}(100,105)(80,-25)
\setlength{\unitlength}{1pt}

\put(110,0){
\put(0,0){\vector(1,0){160}}
\put(160,8){\makebox(0,0){$\rho$}}
\put(0,0){\vector(0,1){80}}
\put(-3,70){\makebox(0,0)[r]{$D$}}

\put(150,-3){\makebox(0,0)[t]{$1$}}
\put(150,-2){\line(0,1){4}}

\put(0,0){\thicklines{\qbezier(0,0)(40,60)(70,0)}}
\put(0,0){\thicklines{\qbezier(70,0)(88,-25)(110,0)}}
\put(0,0){\thicklines{\qbezier(110,0)(130,30)(150,0)}}

\put(70,-4){\makebox(0,0)[rt]{$\alpha$}}
\put(110,-3){\makebox(0,0)[t]{$\beta$}}

\put(80,-20){\makebox(0,0)[t]{{\em (a)}}}
}

\put(330,0){
\put(0,0){\vector(1,0){160}}
\put(160,8){\makebox(0,0){$\rho$}}
\put(0,0){\vector(0,1){80}}
\put(-3,70){\makebox(0,0)[r]{$f$}}

\put(150,-3){\makebox(0,0)[t]{$1$}}
\put(150,-2){\line(0,1){4}}

\put(0,0){\thicklines{\qbezier(0,0)(40,120)(70,50)}}
\put(0,0){\thicklines{\qbezier(70,50)(88,10)(110,30)}}
\put(0,0){\thicklines{\qbezier(110,30)(130,50)(150,0)}}
\put(72,46){\line(5,-2){70}}
\put(72,46){\line(-5,2){43}}

\multiput(70,0)(0,5){9}{$.$}
\put(70,-3){\makebox(0,0)[t]{$\alpha$}}

\put(100,40){\makebox(0,0)[b]{$s_{\ell^\pm}$}}

\multiput(110,0)(0,5){6}{$.$}
\put(110,-3){\makebox(0,0)[t]{$\beta$}}

\multiput(94,0)(0,5){5}{$.$}
\put(94,-3){\makebox(0,0)[t]{$\gamma$}}

\multiput(139,0)(0,4){5}{$.$}
\put(139,-3){\makebox(0,0)[t]{$\ell^+$}}

\multiput(28,0)(0,5){13}{$.$}
\put(28,-3){\makebox(0,0)[t]{$\ell^-$}}
\put(80,-20){\makebox(0,0)[t]{{\em (b)}}}
}

\end{picture}
\caption{\label{f:f2}{{\em (a)}: the diffusivity $D$ in case (D2); {\em (b)}: the flux function $f$.}}
\end{figure}
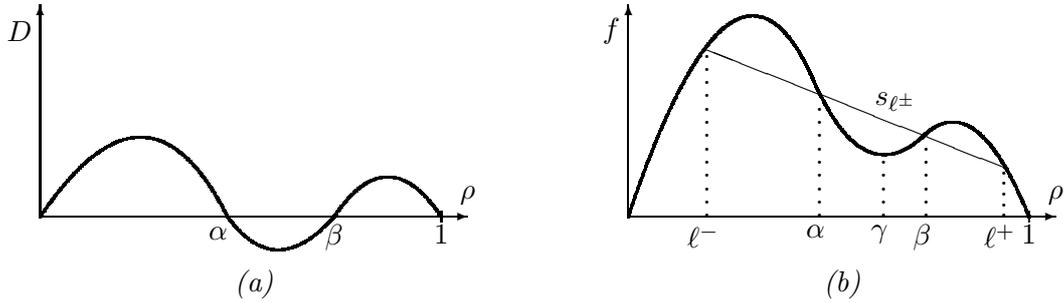

As in the discussion preceding Lemma \ref{l:lemma1}, under \eqref{e:xi0} there are three (unique) values $0\le\xi_1\le\xi_2\le\xi_3$ such that
\begin{equation}\label{e:xi12}
\begin{array}{ll}
\phi(\xi)=\alpha\ \hbox{ for } 0\le\xi\le\xi_1\quad  &\hbox{ and } \quad \alpha<\phi(\xi)<\beta\quad \hbox{ for } \xi_1<\xi<\xi_2,
\\
\phi(\xi)=\beta\ \hbox{ for } \xi_2\le\xi\le\xi_3 \quad  &\hbox{ and } \quad \phi(\xi)>\beta\quad \hbox{ for } \xi>\xi_3,
\end{array}
\end{equation}
see Figure \ref{f:TWW}. A result analogous to Lemma \ref{l:lemma1} can be proved and then, for simplicity, we focus on the case
\begin{equation}\label{e:xi0xi2=0}
\xi_0=\xi_1=0, \quad \xi_2=\xi_3.
\end{equation}
For brevity, in the next statement we collect only the most important facts about case (D2); other results as in Theorem \ref{t:main} follow in a direct way. The proof is omitted.

\begin{figure}[htbp]
\begin{picture}(100,125)(80,-15)
\setlength{\unitlength}{1pt}

\put(270,0){
\put(0,0){\vector(1,0){190}}
\put(0,0){\line(-1,0){130}}
\put(190,8){\makebox(0,0){$\xi$}}
\put(0,0){\vector(0,1){100}}
\put(-3,100){\makebox(0,0)[r]{$\phi$}}

\put(-2,8){\line(1,0){4}}
\put(-2,8){\makebox(0,0)[r]{$\ell^-$}}
\put(-130,8){\line(1,0){80}}

\put(-2,92){\line(1,0){4}}
\put(-2,90){\makebox(0,0)[r]{$\ell^+$}}
\put(190,92){\line(-1,0){80}}

\put(-2,33){\makebox(0,0)[rb]{$\alpha$}}
\put(-2,60){\line(1,0){4}}
\put(-2,56){\makebox(0,0)[rb]{$\beta$}}

\put(0,0){\thicklines{\qbezier(-130,10)(-30,10)(0,30)}}
\put(0,0){\thicklines{\qbezier(0,30)(30,30)(30,30)}} 
\put(0,0){\thicklines{\qbezier(30,30)(40,50)(60,60)}}
\put(0,0){\thicklines{\qbezier(60,60)(100,60)(100,60)}} 
\put(0,0){\thicklines{\qbezier(100,60)(110,90)(190,90)}}

\put(0,-3){\makebox(0,0)[t]{$\xi_0=0$}}

\multiput(29,0)(0,5){7}{$.$}
\put(30,-3){\makebox(0,0)[t]{$\xi_1$}}

\multiput(59,0)(0,5){12}{$.$}
\put(60,-3){\makebox(0,0)[t]{$\xi_2$}}

\multiput(99,0)(0,5){12}{$.$}
\put(100,-3){\makebox(0,0)[t]{$\xi_3$}}

}

\end{picture}
\caption{\label{f:TWW}{Under (D2), a profile $\phi$ in the case $0<\xi_1<\xi_2<\xi_3$.}}
\end{figure}
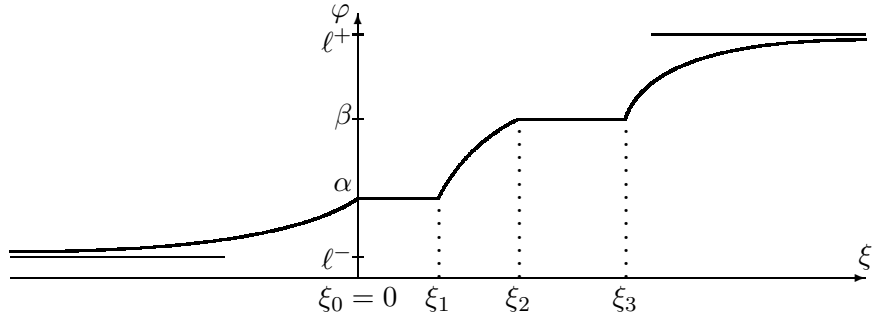

\begin{theorem}\label{t:main2} \ Assume \emph{(f)} and \emph{(D2)}; let $\ell^- \in [0, \alpha)$ and $\ell^+ \in (\beta, 1]$. Equation \eqref{e:E} admits a wavefront solution whose profile $\phi$ satisfies \eqref{e:infty} if and only if
\begin{equation}\label{e:existence wfs2}
\frac{f(\alpha)-f(\ell^-)}{\alpha-\ell^-}=\frac{f(\beta)-f(\alpha)}{\beta-\alpha}= \frac{f(\ell^+)-f(\beta)}{\ell^+-\beta}=:c_{\ell^{\pm}},
\end{equation}
and
\begin{equation}\label{e:existence about f2}
\begin{array}{rll}
&f(\rho)>\ds\frac{f(\alpha)-f(\ell^-)}{\alpha-\ell^-}(\rho-\alpha)+f(\alpha), \quad &\rho \in (\ell^-, \alpha),
\\[3mm]
&f(\rho)<\ds\frac{f(\beta)-f(\alpha)}{\beta-\alpha}(\rho-\beta)+f(\beta), \quad &\rho \in (\alpha, \beta),
\\[3mm]
&f(\rho)>\ds\frac{f(\ell^+)-f(\beta)}{\ell^+-\beta}(\rho-\ell^+)+f(\ell^+), \quad &\rho \in (\beta, \ell^+).
\end{array}
\end{equation}
Assume \eqref{e:xi0xi2=0}. Then the profile is unique;
moreover, $\phi^{\, \prime}(\xi)>0$ when $\ell^-<\phi(\xi)<\ell^+$, $\xi\ne 0$, $\xi\ne\xi_2$.
\end{theorem}

We refer to Figure \ref{f:f2}{\em (b)} for a geometrical interpretation of conditions \eqref{e:existence wfs2} and \eqref{e:existence about f2}.

\smallskip

At last, we consider the family of equations
\begin{equation}\label{e:Eepsilon}
\rho_t + f(\rho)_x=\left(\epsilon D(\rho)\rho_x\right)_x, \qquad t\ge 0, \, x\in \R,
\end{equation}
depending on the parameter $\epsilon \in (0,1]$; we are interested in the limit as $\epsilon \to 0^+$ of the wavefronts profiles $\phi_{\epsilon}$. This subject falls in the much more general issue of the convergence of solutions $\rho_\eps$ of \eqref{e:Eepsilon} to a solution $\rho$ of the conservation law \eqref{e:cl}. If $D>0$, a positive answer is provided in \cite{Kruzhkov}; we refer to \cite[\S 6]{Dafermos} for more information. The case $D\ge0$ was first considered in \cite{Volpert-Hudjaev}; see \cite{CCR} for a short proof in the presence of source terms and updated references. To the best of our knowledge, no analogous information  is known when $D$ changes sign. We provide now a convergence result concerning wavefronts.

About \eqref{e:Eepsilon}, we assume conditions (f), (D1), \eqref{e:existence wfs}, \eqref{e:existence about f}; moreover, for sake of simplicity, we also suppose
\begin{align}\label{e:Dslope}
D^{\, \prime}(\alpha)<0.
\end{align}
By Theorem \ref{t:main}, for every  $\ell^- \in [0, \alpha)$, $\ell^+\in (\alpha, 1]$ and $\epsilon \in (0,1]$, there exist unique profiles $\phi_{\epsilon}(\xi)$ of wavefronts of \eqref{e:Eepsilon} satisfying \eqref{e:xi0xi1=0}. All of them have the same speed $c_{\ell^\pm}$.

\begin{theorem}\label{t:conv left}
We assume \emph{(f)}, \emph{(D1)} and consider $\ell^- \in [0, \alpha)$, $\ell^+\in (\alpha, 1]$; we also assume \eqref{e:existence wfs}, \eqref{e:existence about f} as well as \eqref{e:Dslope}. Let $\phi_\eps$ be the unique profiles of wavefront solutions to \eqref{e:Eepsilon} satisfying \eqref{e:infty} and \eqref{e:xi0xi1=0}. Then
\begin{equation}\label{e:conv left}
\lim_{\eps\to0+}\phi_\eps(\xi)=: \phi_0(\xi)=\left\{
\begin{array}{ll}
\ell^- \hbox{ if }\xi<0,
\\[2mm]
\ell^+ \hbox{ if }\xi>0.
\end{array}
\right.
\end{equation}
The convergence is uniform in every interval $(-\infty, -\delta)$ and $(\delta, +\infty)$ with $\delta>0$.
\end{theorem}

\smallskip

The results of Theorem \ref{t:main} can be easily reformulated to cover the case $\xi_1>0$; in particular formula \eqref{e:slopeLR} corresponds to the limit for $\xi\to0-$, while a completely analogous result holds for the limit $\xi\to\xi_1+$. Assume $D'(\alpha)<0$; if $\xi_1=0$ then $\phi$ is of class $C^1$ at $\xi=0$, by \eqref{e:slopeLR}, while if $\xi_1>0$ this does not hold unless $f'(\alpha)=c_{\ell^\pm}$. Analogous remarks apply to Theorem \ref{t:main2} in the case $\xi_1>0$ and $\xi_2<\xi_3$. About Theorem \ref{t:conv left}, consider again a family $\phi_\eps$ of profiles for \eqref{e:Eepsilon} but suppose that there exists $\xi_1>0$ that does not depend on $\eps$ such that $\alpha=\phi_\eps(\xi_1)<\phi_\eps(\xi)$ for $\xi>\xi_1$. Arguing as in the proof of Theorem \ref{t:conv left} we deduce that
\begin{equation}\label{e:limalpha}
\lim_{\eps\to0+}\phi_\eps(\xi)=:\phi_{0,\xi_1}(\xi)=
\left\{
\begin{array}{ll}
\ell^-&\hbox{ if }\xi\in(-\infty,0),
\\
\alpha&\hbox{ if }\xi\in(0,\xi_1),
\\
\ell^+& \hbox{ if }\xi\in(\xi_1,\infty).
\end{array}
\right.
\end{equation}

\smallskip

Theorem \ref{t:conv left} allows us to rigorously comment on the above results from the hyperbolic point of view $D=0$; we are concerned with equation \eqref{e:cl}. First, we remark that the function $\rho_0(x,t)=\phi_0(x-c_{\ell^\pm}t)$ is a weak (distributional) solution of equation \eqref{e:E} because the Rankine-Hugoniot conditions are satisfied.

Conditions analogous to \eqref{e:existence about f} are well known in the hyperbolic setting \cite[Thm. 4.4]{Bressan}. In particular the discontinuous solution $\rho_0$ is not entropic; it is so the analogous solution joining $\ell^-$ with $\alpha$. Referring for instance to the case depicted in Figure \ref{f:f}{\em (b)}, the Lax inequality $f'(\ell^-)>c_{\ell^\pm}$ is satisfied while the inequality $c_{\ell^\pm}>f'(\ell^+)$ fails: the shock is compressive on the left and undercompressive on the right. However, even if $\rho_0$ is not entropic, Theorem \ref{t:main} shows that it has a viscous profile, where the term \lq\lq viscous\rq\rq\ refers to a negative diffusivity in the nonentropic part of the solution; of course, such a wave is unstable in the sense of  \cite[Rem. 4.7]{Bressan}. Different scenarios are also possible: for instance, the one-sided sonic case $c_{\ell^\pm}=f'(\ell^+)\ne f'(\ell^-)$ (or $c_{\ell^\pm}=f'(\ell^-)\ne f'(\ell^+)$) or even the doubly sonic case $c_{\ell^\pm}=f'(\ell^\pm)$, see Figure \ref{f:ds}. The former case has been considered in \cite{Colombo-Rosini2005} (see cases (R1) and (R3){\em (a)} there) in the framework of nonclassical shocks \cite{LeFloch}. However, generically, there is a sheaf of lines through $\left(\alpha,f(\alpha)\right)$ that intersect the graph of $f$ in two further points and then a one-parameter family of end states $\left(\ell^-,\ell^+\right)$ for which \eqref{e:existence about f} is satisfied.

In the case $\xi_1>0$, the function $\rho_{0,\xi_1}(x,t)=\phi_{0,\xi_1}(x-c_{\ell^\pm}t)$ is a weak solution of the conservation law \eqref{e:cl} because its jumps satisfy the Rankine-Hugoniot conditions. However, the entropy condition fails, as discussed above.

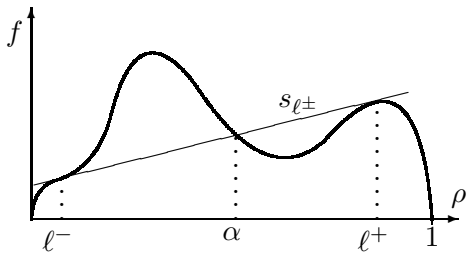
\begin{figure}[htbp]
\begin{picture}(100,100)(80,-15)
\setlength{\unitlength}{1pt}

\put(220,0){
\put(0,0){\vector(1,0){160}}
\put(160,8){\makebox(0,0){$\rho$}}
\put(0,0){\vector(0,1){80}}
\put(-3,70){\makebox(0,0)[r]{$f$}}

\put(150,-3){\makebox(0,0)[t]{$1$}}
\put(150,-2){\line(0,1){4}}

\put(0,0){\thicklines{\qbezier(0,0)(1,13)(10,15)}}
\put(0,0){\thicklines{\qbezier(10,15)(25,20)(30,40)}}
\put(0,0){\thicklines{\qbezier(30,40)(40,80)(62,50)}}
\put(0,0){\thicklines{\qbezier(62,50)(88,10)(110,30)}}
\put(0,0){\thicklines{\qbezier(110,30)(145,70)(150,0)}}

\put(1,13){\line(4,1){140}}

\multiput(75,0)(0,5){7}{$.$}
\put(75,-3){\makebox(0,0)[t]{$\alpha$}}

\put(100,40){\makebox(0,0)[b]{$s_{\ell^\pm}$}}

\multiput(10,0)(0,5){3}{$.$}
\put(10,-3){\makebox(0,0)[t]{$\ell^-$}}

\multiput(128,0)(0,5){9}{$.$}
\put(128,-3){\makebox(0,0)[t]{$\ell^+$}}

}

\end{picture}
\caption{\label{f:ds}{A flux function $f$ in the doubly sonic case.}}
\end{figure}

About Theorem \ref{t:main2}, consider for instance the case depicted in Figure \ref{f:f2}{\em (b)}. The corresponding discontinuous solution $\rho$ joining $\ell^-$ on the left to $\ell^+$ on the right with a jump propagating with velocity $c_{\ell^\pm}$ is still nonentropic. Both Lax inequalities $f'(\ell^-)>c_{\ell^\pm}>f'(\ell^+)$ are now satisfied; this does not imply that the solution is entropic because the flux is not convex \cite[Remark 4.7]{Bressan}. Also shock waves connecting the states $\ell^-$ and $\ell^+$ as in Figure \ref{f:f2}{\em (b)} have been considered in \cite{Colombo-Rosini2005} (see case (R3){\em (b)} there). Notice that in the framework of Theorem \ref{t:main2} we generically have a unique pair of end states $\ell^\pm$.

In \cite{Colombo-Rosini2005} the (hyperbolic) zone of \lq\lq panic\rq\rq\ is modeled by the interval $(\gamma,1]$, where $\gamma$ is a local minimum of $f$, see Figure \ref{f:f2}{\em (b)}; the zone of aggregation (where $D<0$) is $(\alpha,\beta)$. The balance between panic and aggregation can be explained by our diffusive model as follows. When the density is in the interval $(\alpha,\gamma)$, panic has not yet emerged but the crowd shows an aggregative behavior to face the perceived danger. This behavior persists even after the threshold $\gamma$ is trespassed if the density is not exceeding, namely, in the interval $(\gamma,\beta)$. Values of $\rho\in(\beta,1)$ are unbearable and push the crowd to (slightly) diffuse again.


\section{Applications to collective movements}\label{sec:ex}
\setcounter{equation}{0}
In this section we provide some examples concerning the modeling of vehicular traffic flows or crowds dynamics. In these cases assumption (f) specializes to \cite{Lighthill-Whitham, Richards}
\begin{itemize}
\item[{(fcm)}]
$f(\rho) = \rho v(\rho)$, with $v\in C^1[0,1]$, $v(\rho)\ge0$ for $\rho\in[0,1)$ and $v(1)=0$.
\end{itemize}
Here $v$ represents the velocity; from a modeling point of view, in the interval $[0,1)$ it may vanishes once \cite{Colombo-Rosini2005} and is decreasing at least in a right neighborhood of $0$. The main issue regards the choice of $D$; the properties $D(0)=D(1)=0$ would be desirable \cite{Bellomo-Delitala-Coscia, BTTV}. We now list the main models of $D$ occurring in the literature.

In \cite{Nelson_2000} the author proposed the expression
\begin{equation}\label{e:DN}
D(\rho)= -\delta \rho v'(\rho) -\tau\rho^2v'(\rho)^2 = -\rho v'(\rho)\left(\delta+\tau\rho v'(\rho)\right),
\end{equation}
where $\delta>0$ is an anticipation distance and $\tau>0$ a relaxation time; see also \cite{Herty-Illner} for an analogous introduction of these parameters in a kinetic framework. The ratio $\delta/\tau = :v^s$, which occurs several times in the following, represents the velocity needed to cover the distance $\delta$ in the time $\tau$; it can be understood as a {\em safety velocity}. Under this notation \eqref{e:DN} can be written as
\begin{equation}\label{e:DNvs}
D(\rho)= -\tau\rho v'(\rho)\left(v^s + \rho v'(\rho)\right).
\end{equation}
If $\ov=\max_{\rho\in [0,1]}v(\rho)$, then a natural requirement is
\begin{equation}\label{e:vvs}
\ov\le v^s.
\end{equation}

In the case of crowds dynamics, the parameter $\tau$ is very small and may be dropped \cite{BTTV}; this leads to
\begin{equation}\label{e:DT}
D(\rho)= -\delta \rho v'(\rho).
\end{equation}
The case when $\delta$ in \eqref{e:DN} depends on $\rho$ is also proposed in \cite{Nelson_2000}, where $\delta(\rho) = h v^2(\rho)$ for $h>0$. If it is so, \eqref{e:DN} becomes
\begin{equation}\label{e:Dgamma}
D(\rho) = -\rho v'(\rho)\left(h v^2(\rho) +\tau\rho v'(\rho)\right).
\end{equation}
In the case of pedestrian flows, instead of \eqref{e:Dgamma} one may consider \cite[Figure 4]{BTTV}
\begin{equation}\label{e:Dgammap}
D(\rho) = -\rho v'(\rho)\left(h v(\rho) +\tau\rho v'(\rho)\right).
\end{equation}
For simplicity, in \eqref{e:Dgamma} and \eqref{e:Dgammap} we have taken $\tau\ge0$ to be independent from $\rho$. 

We also mention that several different models for $D$ follow by a kinetic approximation \cite{Herty-Puppo-Roncoroni-Visconti}. In that paper the authors motivate the occurrence of stop and go waves precisely by the presence of zones with negative diffusivities. For instance, in the case of a kinetic model with two microscopic velocities $0\le\xi_1<\xi_2$, one deduces from \cite{Herty-Puppo-Roncoroni-Visconti} that $D(\rho) = \tau\left((\xi_1+\xi_2)q'(\rho)+\xi_1\xi_2-[q'(\rho)]^2\right)$, where $q(\rho)=\rho v(\rho)$ is the hyperbolic flow and $\tau>0$ a reaction time. In our model $v$ ranges from $0$ to $\ov$ and then we take $\xi_1=0$, $\xi_2=\ov$; we deduce
\begin{equation}\label{e:Dkin}
D(\rho) = \tau\left(\rho v'(\rho)+v(\rho)\right)\left(\ov-\rho v'(\rho)-v(\rho)\right),
\end{equation}
which must be compared with \eqref{e:DNvs}. For simplicity, we limit the examples below to the simpler diffusivities \eqref{e:DN}, \eqref{e:DT}, \eqref{e:Dgamma}, \eqref{e:Dgammap}.

\subsection{The case (D1)}\label{sub_D1}
In this subsection we investigate when assumptions (fcm), (D1), \eqref{e:existence about f} are satisfied according to the choice of $v$, when $D$ is chosen as in \eqref{e:DN}--\eqref{e:Dgammap}. We begin with some {\em negative} results; some proofs are deferred to Section \ref{s:proof}.

\begin{lemma}\label{l:noD1}
If $D(\rho)=-a(\rho)v'(\rho)$ with $a>0$ in $(0,1)\setminus\{\alpha\}$, then assumptions {\rm (fcm)} and {\rm (D1)} cannot hold together.
\end{lemma}
\begin{proof}
By (D1) we have $v'(\rho)<0$ if $\rho\in(0,\alpha)$ and $v'(\rho)>0$ if $\rho\in(\alpha,1)$. Since $v$ increases in $(\alpha,1)$, by $v(1)=0$ we deduce $v(\alpha)<0$, which contradicts the positivity of $v$.
\end{proof}

Then neither \eqref{e:DT} nor \eqref{e:Dgamma} with $\tau=0$ satisfy both (fcm) and (D1). We consider \eqref{e:DN}.

\begin{lemma}\label{l:D1no}
Assume $D$ is given by \eqref{e:DN}. If $f$ and $D$ satisfy {\rm (fcm)} and {\rm (D1)}, respectively, then $v$ must be decreasing. The simple expressions
\begin{equation}\label{e:vW}
v(\rho)=\frac{\delta}{\alpha\tau}(1-\rho),\quad
v(\rho)
= \frac{\delta}{2\alpha\tau}(1-\rho)\left(1+2\alpha-\rho\right).
\end{equation}
imply both {\rm (fcm)} and {\rm (D1)}. Conditions \eqref{e:existence about f} and \eqref{e:vvs} fail in both cases.
\end{lemma}
The two previous lemmas shows that the expressions \eqref{e:DN}, \eqref{e:DT} and \eqref{e:Dgamma} with $\tau=0$ never or difficultly match conditions (fcm), (D1) and \eqref{e:existence about f} for some simple velocity $v$. Then, we focus on the case \eqref{e:Dgamma} with $\tau>0$, where we provide {\em positive} results. In order that $D(1)=0$ holds, we need $v$ vanishes at second order at $\rho=1$; then, we consider
\begin{equation}\label{e:v3}
v(\rho) = \ov(1-\rho)^2,
\end{equation}
for $\ov>0$. We deduce
\begin{equation}\label{e:SuperD}
D(\rho) = 2h\ov^3\rho(1-\rho)^2\left[(1-\rho)^3-\sigma\rho\right], \qquad \sigma :=\frac{2\tau}{h\ov}>0.
\end{equation}

\begin{lemma}\label{l:quadraticv}
Assume $v$ is given by \eqref{e:v3} and consequently $D$ by \eqref{e:Dgamma} with $\tau\ne0$, see \eqref{e:SuperD}. Then $D$ satisfies {\rm (D1)} for any positive $\ov,h,\tau$ and $\alpha=\alpha(\ov,h,\tau)$ satisfies
\begin{equation}\label{e:sigma}
(1-\alpha)^3 = \sigma \alpha.
\end{equation}
Choose $\tau, h, \ov$ such that $\alpha(\ov,h,\tau)\in(\frac12,1)$; then there are infinitely many pairs $(\ell^-,\ell^+)$ such that \eqref{e:existence wfs} and \eqref{e:existence about f} hold.
\end{lemma}
The point $\alpha$ does not need to be an inflection point of $f$; this happens, in the example above, if $\alpha=\frac23$. Figure \ref{f:fN} gives an illustration of the example considered in Lemma \ref{l:quadraticv} to real-world data. Here we use dimensional variables $0\le\rho\le\orho$, $v(\rho)=\ov (\orho-\rho)^2$ and $f$, $D$ are defined by (fcm), \eqref{e:Dgamma}, respectively. As in \cite{Nelson_2000}, we take as maximal density $\orho=150$ cars per km, maximal velocity $\overline{v}=130 \ {\rm km/h}$, $\tau = 2\ {\rm s}$ and $h=1/15800$.
\begin{figure}[htbp] 
  \centering
  \begin{tabular}{ccc}
  \includegraphics[width=7.5cm]{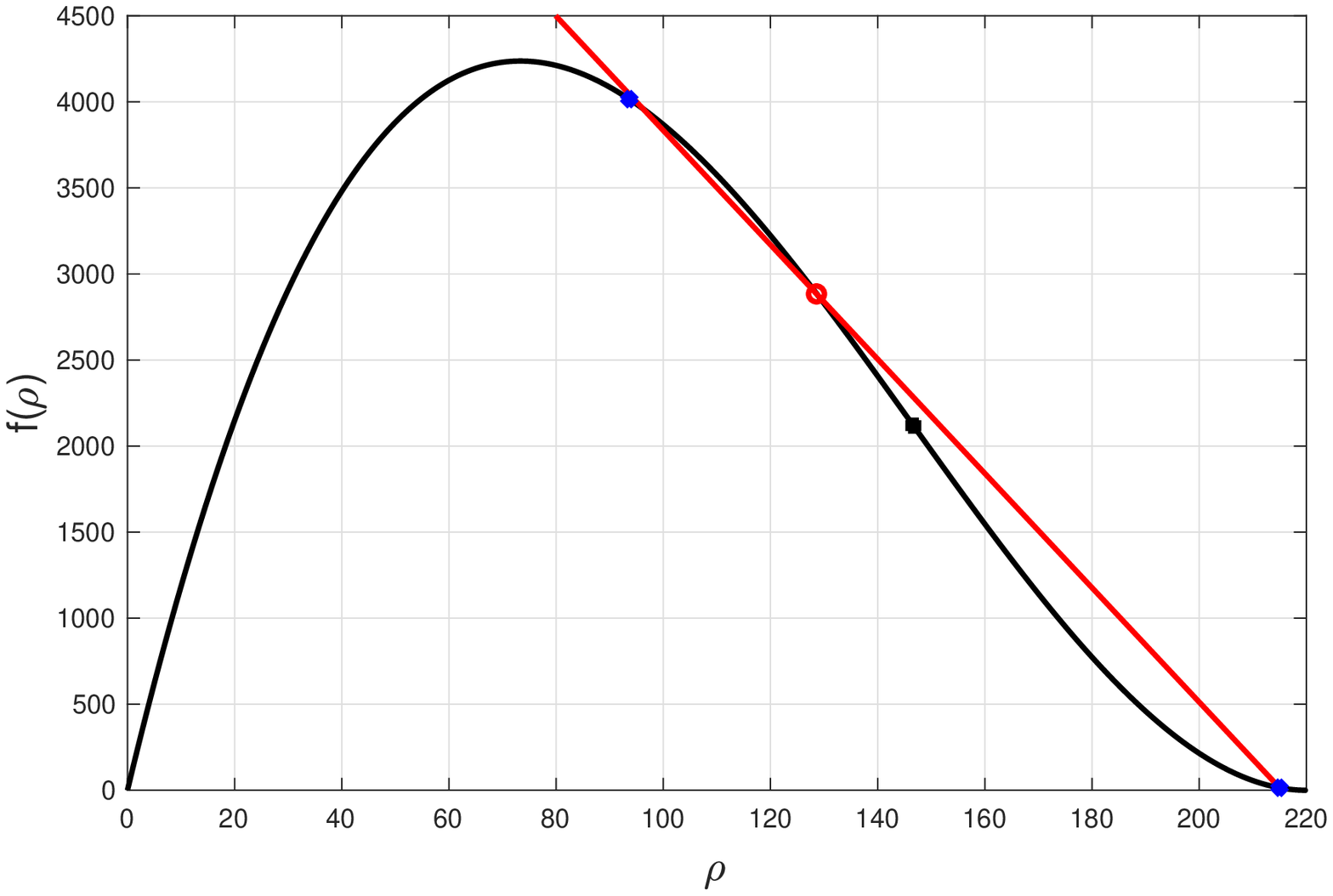}
  &
  \includegraphics[width=7.5cm]{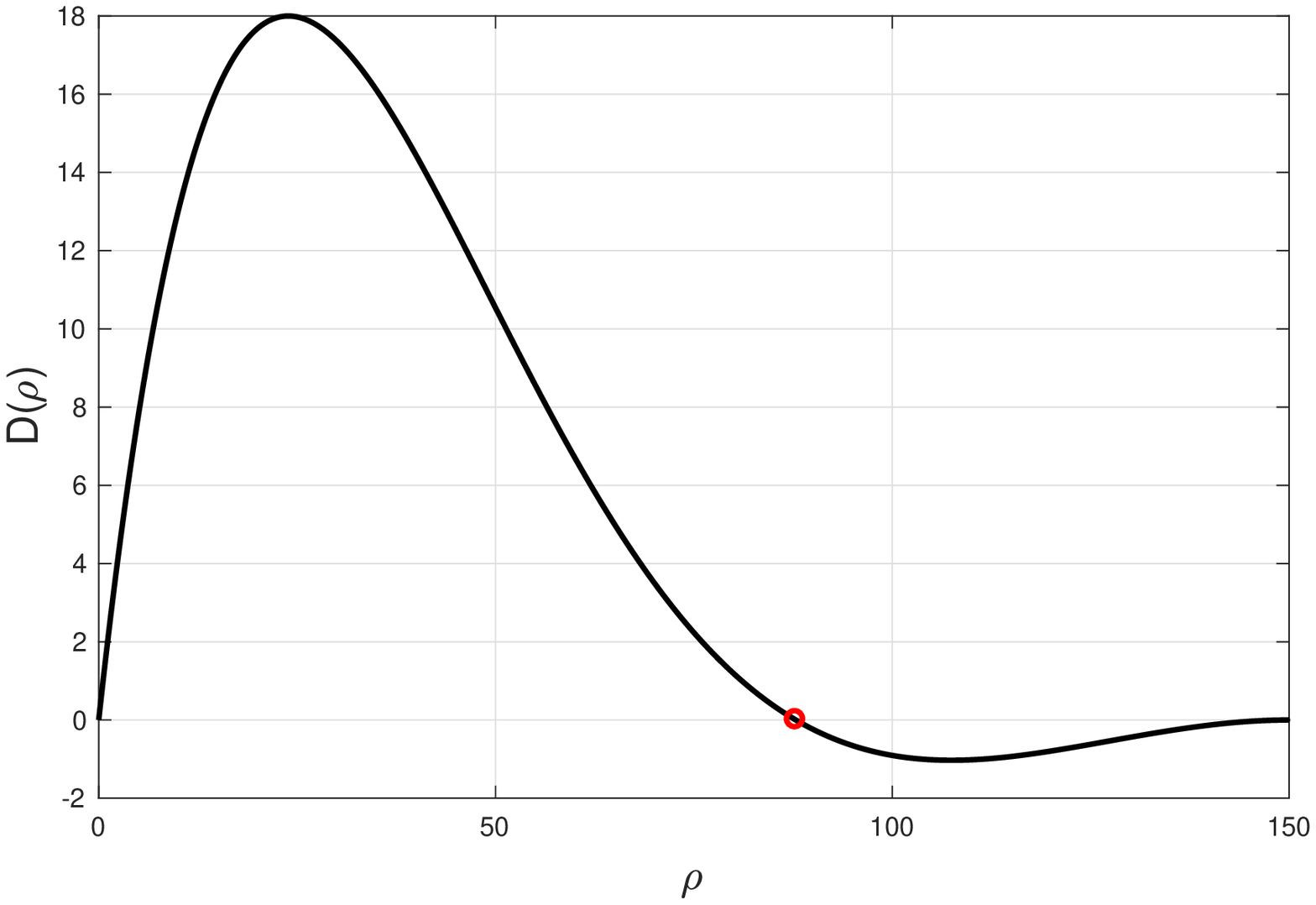}
\end{tabular}
\caption{Plots of flow and diffusivity for $v(\rho)=\ov (\orho-\rho)^2$ and $D$ as in \eqref{e:Dgamma}. Here $\orho=150$ cars/km, $\overline{v}=130 \ {\rm km/h}$, $\tau = 2\ {\rm s}$, $h=1/15800$, see \cite{Nelson_2000}. An empty circle localizes $\alpha\sim88$, a full circle the inflection point of $f$, which is 100. For $\ell^+=147$ we find $\ell^-\sim 65$.}
\label{f:fN}
\end{figure}

The previous example can be generalized to $v(\rho)=\ov(1-\rho^p)^q$, for $p>0$, $q>1$. The condition $q>1$ is needed in order that $f$ has an inflection point (at $\rho = (\frac{1+p}{1+pq})^{1/p}$) in $(0,1)$.

\smallskip

We now consider two laws for pedestrian flows \cite{Venuti-Bruno}. In normalized variables they can be written as
\begin{equation}\label{e:2v}
v(\rho)  = \ov\left(1-e^{\gamma\left(1-\frac1\rho\right)}\right)
\quad
\hbox{ and }
\quad
v(\rho) = \left\{
\begin{array}{ll}
\ov & \hbox{ if }\rho \le a,
\\
\ov e^{\gamma\frac{a-\rho}{1-\rho}} & \hbox{ if }\rho>a,
\end{array}
\right.
\end{equation}
where $\gamma>0$, $\ov>0$ and $0\le a<1$ is a critical density that separates free from congested flow. Both functions are extended by continuity at $0$ and $1$, respectively. The law $\eqref{e:2v}_1$ is called Kladek formula. The law $\eqref{e:2v}_2$ with $a>0$ does not satisfies (fcm) because $v\notin C^1$, unless we consider it only in $(a,1)$ or we set $a=0$; the latter case is also used to model vehicular flows.

About $\eqref{e:2v}_1$, we notice that in this case $f$ is strictly concave; then \eqref{e:existence about f} cannot hold and hence we focus on case $\eqref{e:2v}_2$ from now on. In that case, one easily proves that if $a\gamma<2$ then $f$ is concave in $[0,\tilde\rho)$ and convex in $(\tilde\rho,1]$ for
\begin{equation}\label{e:tilderho}
\tilde\rho = \frac{1}{1+\frac\gamma2(1-a)}\in(a,1).
\end{equation}

If $D$ is as in \eqref{e:DN}, then $D(1)=0$ but $D$ is strictly positive in a left neighborhood of $0$; hence, condition (D1) is not satisfied.
%
If $D$ is as in \eqref{e:DT}, then we have $D(0)=D(1)=0$ and $D(\rho)>0$ if $\rho\in(0,1)$.

Assume $D$ is as in \eqref{e:Dgamma}. Then $D(1)=0$ and $D$ vanishes at the point $\alpha$ defined by
\[
\frac{\tau\gamma(1-a)}{h\ov}\frac{\alpha}{(1-\alpha)^2}=e^{\gamma\frac{a-\alpha}{1-\alpha}},
\]
provided that $a<\frac{1}{1+\bar\sigma}$, for $\bar\sigma= \tau\gamma/(h\ov)$. In this case assumptions (f) and (D1) are satisfied and conditions \eqref{e:existence wfs}, \eqref{e:existence about f} can be numerically checked. We refer to Figure \ref{f:BC} for the plots of flow and diffusivity in the case of real-world data.

\begin{figure}[htbp]	
  \centering
  \begin{tabular}{cc}
  \includegraphics[width=7.5cm]{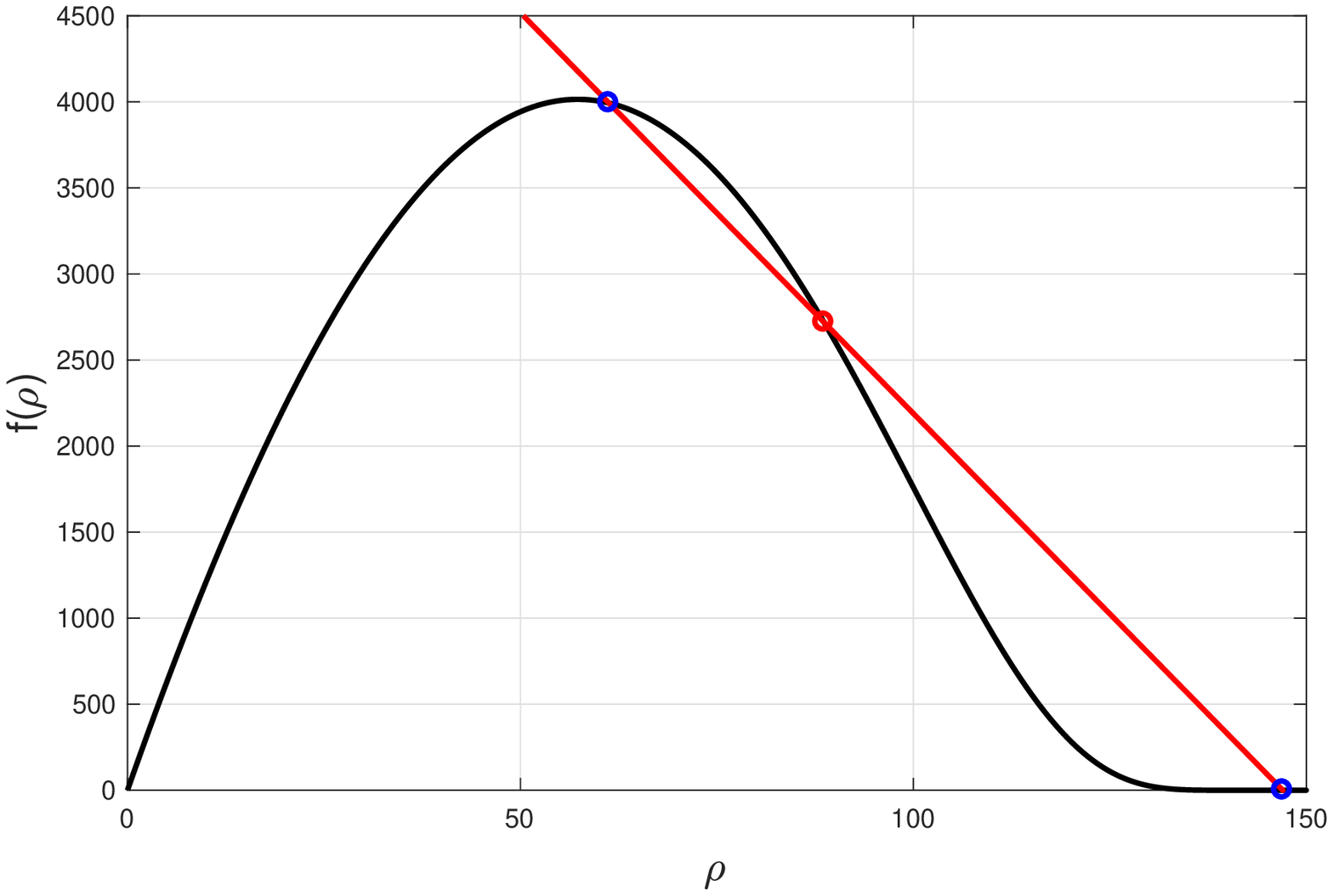}
  &
  \includegraphics[width=7.5cm]{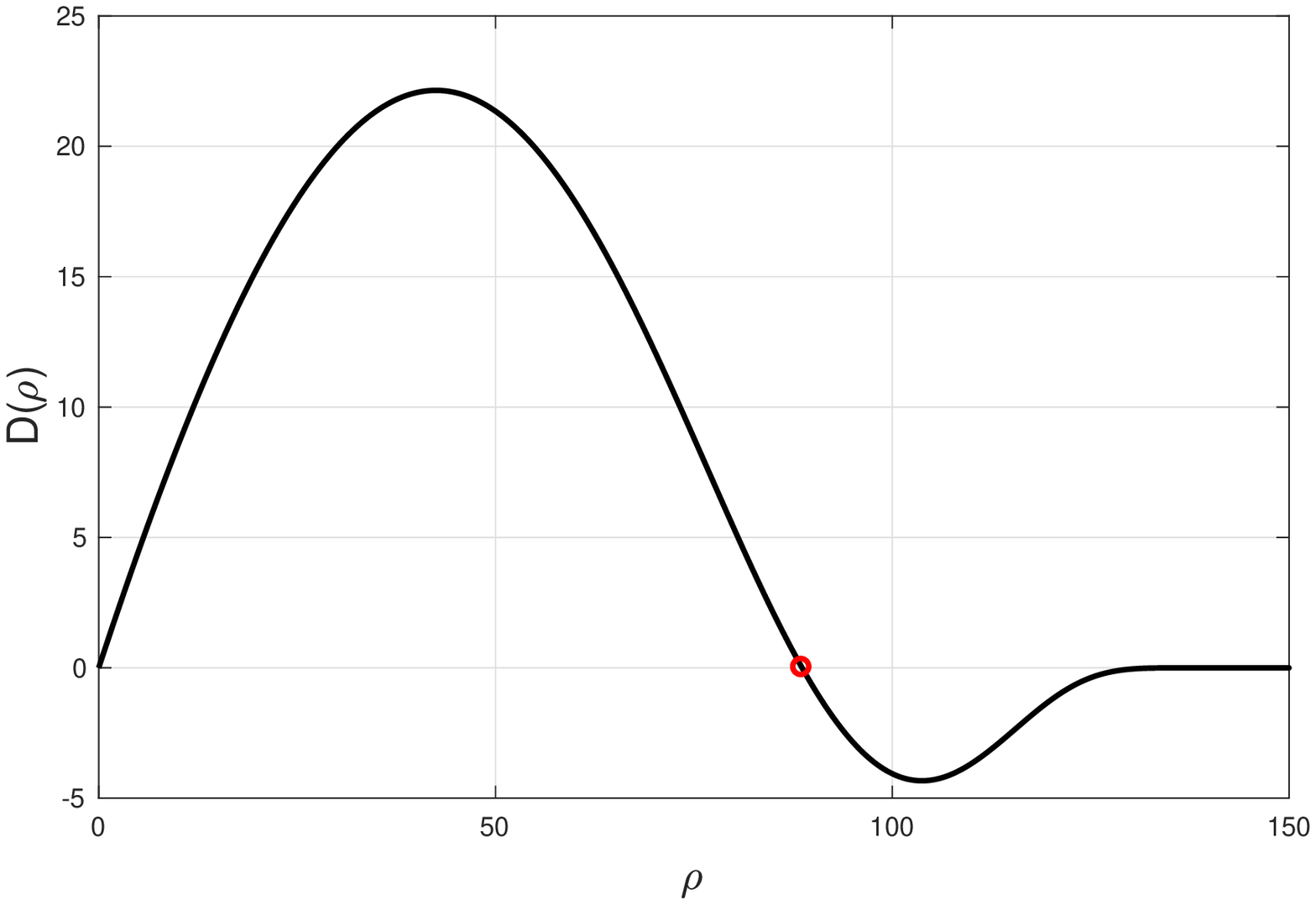}
\end{tabular}
\caption{Plots of  flow and diffusivity for $v$ given by $\eqref{e:2v}_2$ with $a=0$ and $D$ as in \eqref{e:Dgamma}. Data are as in Figure \ref{f:fN}, $\gamma=1$; here $\alpha\sim89$. For $\ell^+=147$ we find $\ell^-\sim 61$.}
\label{f:BC}
\end{figure}

At last, consider again $v$ given by $\eqref{e:2v}_2$ with $a=0$ but $D$ as in \eqref{e:Dgammap}. In this case, for real-world data, conditions \eqref{e:existence wfs} and \eqref{e:existence about f} are satisfied if $\tau$ is sufficiently small, see Figure \ref{f:BCT}. For higher values of $\tau$ the point $\alpha$ falls on the left of the maximum point of $f$.

\begin{figure}[htbp]	
  \centering
  \begin{tabular}{cc}
  \includegraphics[width=7.5cm]{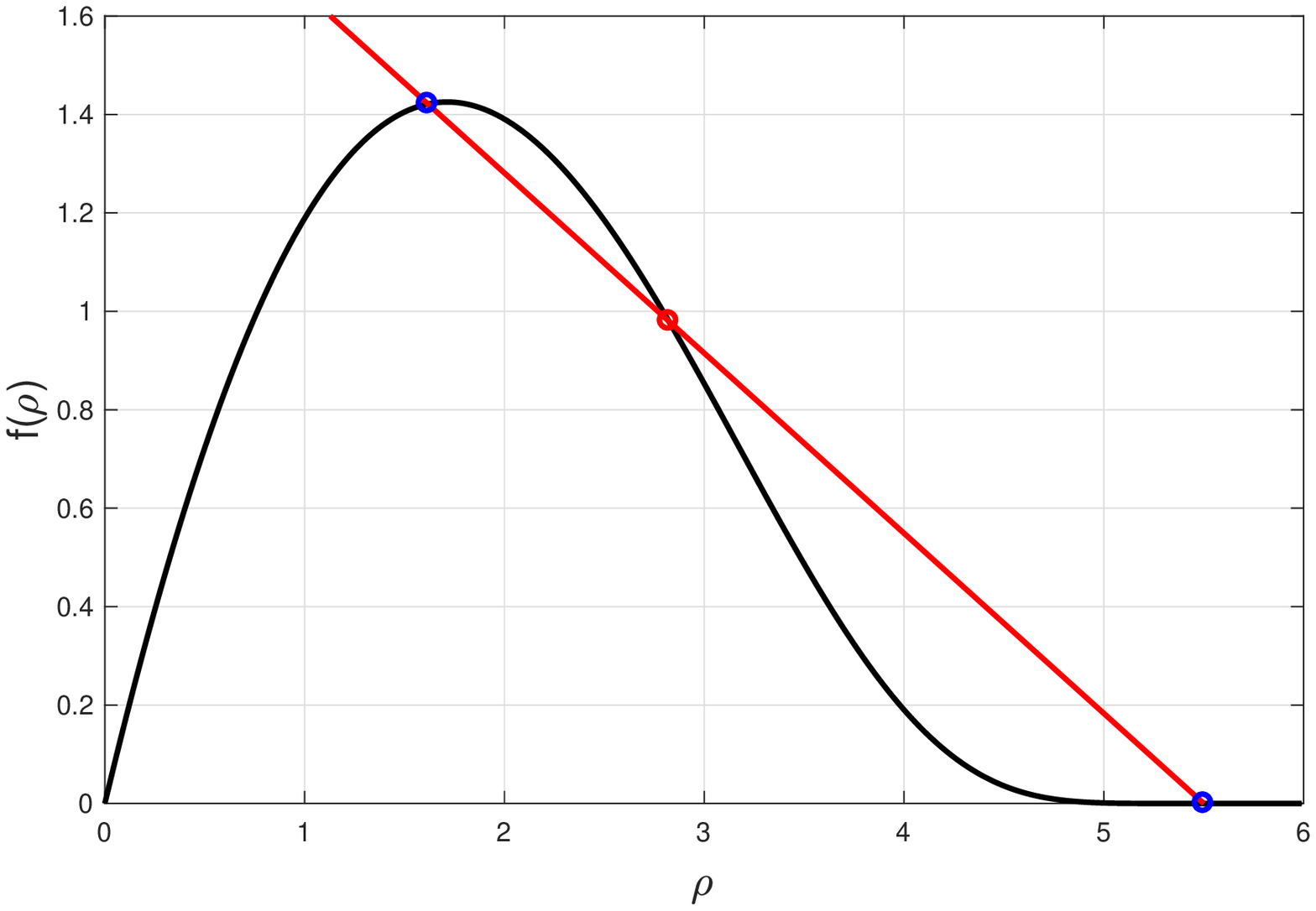}
  &
 \includegraphics[width=7.5cm]{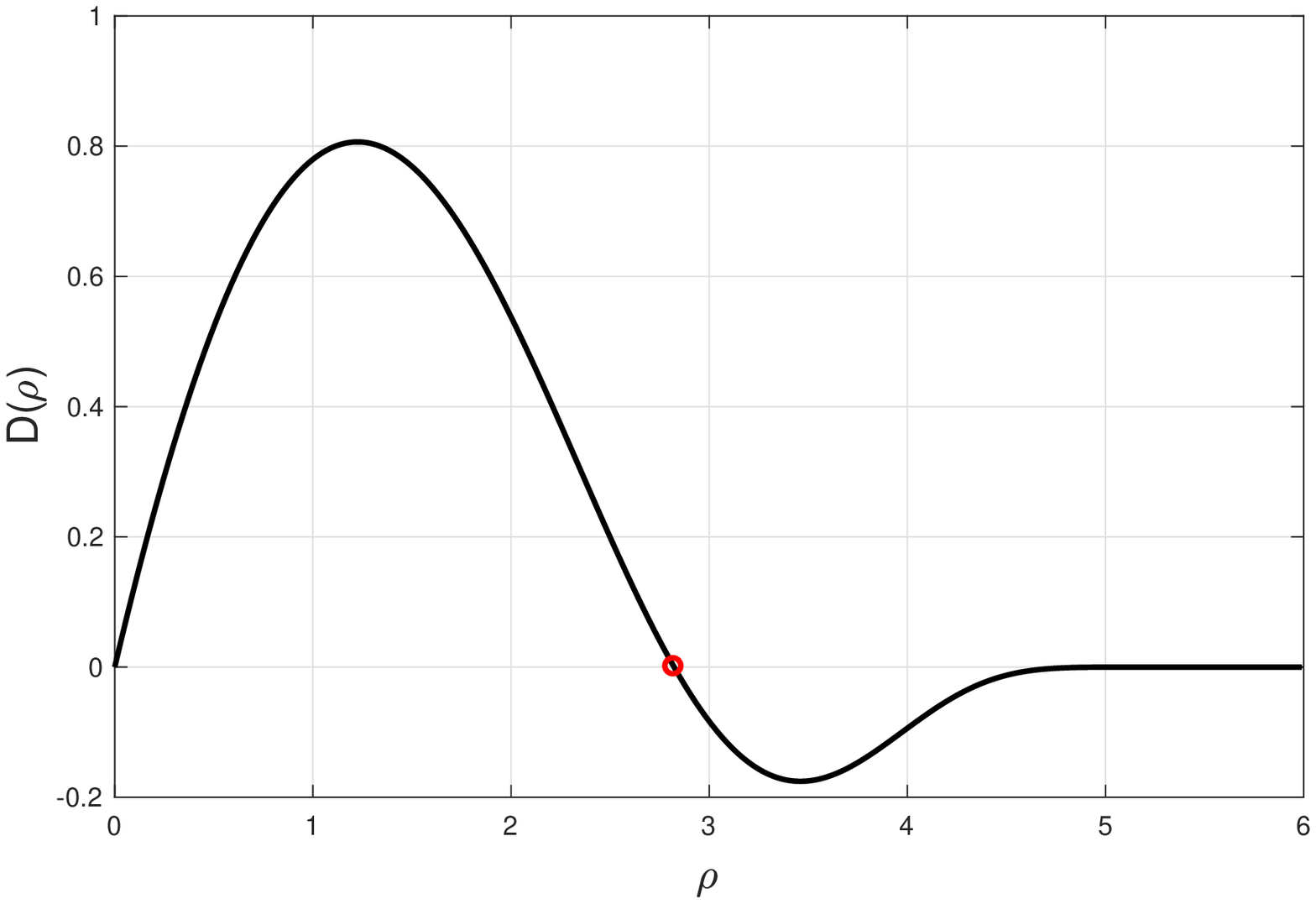}
\end{tabular}
\caption{Plots of flow and diffusivity for $v$ as in $\eqref{e:2v}_2$, $a=0$, and $D$ as in \eqref{e:Dgammap}. Here $\orho=6$ pedestrian/m, $\overline{v}=1.7 \ {\rm m/s}$, $\tau = 0.5\ {\rm s}$, $h=1.5$, $\gamma=1.788$. Data are taken from \cite{Venuti-Bruno}, in the case of rush-hours pedestrian flows.}
\label{f:BCT}
\end{figure}

\subsection{The case (D2)}
We now focus on (D2). As one may guess from the (positive and negative) examples provided in case (D1), in case (D2) it is hard to provide explicit and realistic expressions of velocity laws $v(\rho)$ such that (fcm), (D2) and \eqref{e:existence wfs2}, \eqref{e:existence about f2} are satisfied when $D$ is chosen from one of the simple expressions \eqref{e:DN}--\eqref{e:Dgammap}. We point out this difficulty by considering just one case, analogous to the one in Lemma \ref{l:noD1}. Assume
\begin{equation}\label{e:Dalphabeta}
D(\rho)=-a(\rho)v'(\rho),\quad \hbox{$a>0$ in $(0,1)\setminus\{\alpha,\beta\}$}.
\end{equation}
This covers both \eqref{e:DT} and \eqref{e:Dgamma} with $\tau=0$. The simplest prototype for $v$ in order that (D2) holds and $v(1)=0$ is that $v$ is a third-order polynomial vanishing at $1$, namely
\begin{equation}\label{e:vv}
v'(\rho) = -(\rho-\alpha)(\rho-\beta), \quad v(1) =0.
\end{equation}
For simplicity we dropped any positive multiplicative constant in the right-hand side of \eqref{e:vv}. 
%
%
%
%

\begin{lemma}\label{l:Q}
Let $D$ be as in \eqref{e:Dalphabeta} with $v$ as in \eqref{e:vv}. Then assumptions {\rm (fcm)}, {\rm (D2)} hold if and only if $(\alpha,\beta)\in (0,1)^2$ satisfy $3\beta-2<\alpha<\beta$.
In this case the line $s_{\ell^\pm}$ has positive slope.
\end{lemma}

Lemma \ref{l:Q} shows that, if $D$, $v$ are as in \eqref{e:Dalphabeta}, \eqref{e:vv}, respectively, and assumption (D2) is satisfied, then the case depicted in Figure \ref{f:f2}{\em (b)} (with a line $s_{\ell^\pm}$ having negative slope) never takes place. In other words, this diffusivity cannot provide viscous profiles to  the shock waves introduced in \cite{Colombo-Rosini2005}, case (R3){\em (b)}.


\section{A preliminary result}\label{s:preliminary}
\setcounter{equation}{0}

In this section we state a result about the existence of wavefront solutions for \eqref{e:E} in intervals $[a,b]$ where the diffusivity $D$ has {\em constant} sign. In the case $D$ is positive this result is well known \cite{GK} and the case when $D$ is negative is deduced from the previous one. A short proof of both cases is provided for completeness. The line
\begin{equation*}
s(\rho):=\frac{f(b)-f(a)}{b-a}(\rho-b)+f(b), \qquad \rho \in \R
\end{equation*}
between $\left(a, f(a)\right), \, \left(b, f(b)\right)$  plays a fundamental role in this discussion. Indeed, the existence of wavefront solutions with profile $\phi$ from $a$ to $b$ depends both on the sign of $D$ and on the behavior of the graph of $f(\rho)$ with respect to the line $s(\rho)$. Moreover  $\phi$ satisfies either the boundary conditions
\begin{equation}\label{e:incr}
\phi(-\infty)=a, \qquad \phi(+\infty)=b
\end{equation}
or the opposite ones
\begin{equation}\label{e:decr}
\phi(-\infty)=b, \qquad \phi(+\infty)=a.
\end{equation}
We use the notation
\begin{equation*}
J:=\left\{\xi \in \R \, : \, a <\phi(\xi) <b\right\}.
\end{equation*}

\begin{theorem}\label{t:wfsConst} \ Let $f, \, D \in C^1[a,b]$. Assume
\begin{enumerate}[{(a)}]
\item $D(\rho)>0$ in $(a,b)$; then
\begin{itemize}
\item[{(a1)}] equation \eqref{e:E} has a wavefront of profile $\phi$ with \eqref{e:incr} iff  $f(\rho)> s(\rho)$ for $ \rho \in (a,b)$;
\item[{(a2)}] equation \eqref{e:E} has a wavefront of profile $\phi$ with \eqref{e:decr} iff $f(\rho)< s(\rho)$ for $\rho \in (a,b)$;
\end{itemize}
 \item $D(\rho)<0$ in $(a,b)$; then
 \begin{itemize}
\item[{(b1)}] equation \eqref{e:E} has a wavefront of profile $\phi$ with \eqref{e:decr} iff $f(\rho)> s(\rho)$ for $\rho \in (a,b)$;
\item[{(b2)}] equation \eqref{e:E} has a wavefront of profile $\phi$ with \eqref{e:incr} iff $f(\rho)< s(\rho)$ for $\rho \in (a,b)$.
\end{itemize}
 \end{enumerate}
The profile $\phi$ is unique (up to shifts), $\phi \in C^2(J)$ and $\phi^{\, \prime}(\xi) > 0$ ($\phi^{\, \prime}(\xi) < 0$) for $\xi \in J$ in cases (a1) and (b2) (resp., in cases (a2) and (b1));
the wave speed is
\begin{equation}\label{eq:speed}
c = \dfrac{f(b) - f(a)}{b-a}.
\end{equation}
Moreover, the following holds true.
\begin{itemize}
\item[{(i)}] $\phi$ is sharp in $a$ if and only if $D(a)=0.$
In this case there exists $\xi_a$ such that $\phi(\xi)=a$ for $\xi\le\xi_a$ (for $\xi\ge\xi_a$) in cases (a1), (b2) (resp., in cases (a2), (b1)), $\phi$ is not differentiable in $\xi_a$ and $\lim_{\xi \to \xi_a}D\left(\phi(\xi)\right)\phi^{\prime}(\xi) = 0$.

\item[{(ii)}] $\phi$ is sharp in $b$ if and only if $D(b)=0.$
In this case there exists $\xi_b$ such that $\phi(\xi)=b$ for $\xi\ge\xi_b$ (for $\xi\le\xi_b$) in cases (a1), (b2) (resp., in cases (a2), (b1)), $\phi$ is not differentiable in $\xi_b$ and $\lim_{\xi \to \xi_b}D\left(\phi(\xi)\right)\phi^{\prime}(\xi) = 0$.

\item[{(iii)}] In the other cases $J=\mathbb{R}$  and $\lim_{\xi \to \pm\infty}\phi^{\prime}(\xi) = 0$.
\end{itemize}
\end{theorem}

\begin{proof} \ We first prove the cases {\em (a1)}, {\em (a2)} by means of \cite[Theorem 9.1]{GK}; then, we reduce items {\em (b1)}, {\em (b2)} to {\em (a1)}, {\em (a2)}, respectively, by a suitable change of variables.

Preliminarly, consider equation \eqref{e:rtR}
for $g\in C^1[0,1]$ and $E\in C^1[0,1]$. When $E>0$ in $(0,1)$, and assuming $g(0)=0$ without any loss of generality, by \cite[Theorem 9.1]{GK} there exists a wavefront from $1$ to $0$ of equation \eqref{e:rtR} if and only if $g(r)/r<g(1)$ for every  $r\in(0,1)$; the speed of the wavefront is $g(1)$.

\smallskip
\noindent \emph{(a1)} \ We denote
\begin{equation}\label{e:GE1}
E(r):=D\left(b-(b-a)r\right),  \quad g(r):=-\dfrac{f\left(b-(b -a)r\right)-f(b)}{b-a}, \quad   r\in[0,1].
\end{equation}
Notice that $E>0$ in $(0,1)$ and $g(0)=0$. The condition $f(\rho)>s(\rho)$, $\rho \in (a,b)$, is equivalent to $g(r)/r<g(1)$, $r\in(0,1)$; if it holds, by \cite[Theorem 9.1]{GK} equation \eqref{e:rtR} with $g$ and $E$ as in \eqref{e:GE1} has a wavefront with profile $\psi$ such that $\psi \in C^2(I)$, where
\begin{equation}\label{e:defI}
I:=\{\xi \in \R \, :\, 0<\psi(\xi) <1\}.
\end{equation}
Moreover, the following hold true:
\begin{equation}\label{e:psi1}
\psi(-\infty) =1, \ \psi(+\infty)=0,\qquad \psi^{\, \prime}(\xi) <0 \text{ for } \xi \in I, \qquad c = \dfrac{f(b) - f(a)}{b-a}.
\end{equation}
At last, by Remark \ref{r:ac}, $\psi$ satisfies
\begin{equation*}
\left( E(\psi)\psi^{\prime}\right)^{\prime}+(c\psi -g(\psi))^{\prime}=0 \quad \text{ in } I.
\end{equation*}
Then we define $\phi(\xi):=b-(b-a)\psi(\xi)$, $\xi \in \R$. It is straightforward to show that $\phi \in C^2(J)$, it satisfies \eqref{e:ODE} in $J$, \eqref{e:incr} and $\phi^{\prime}(\xi) >0$ for $\xi \in J$. Moreover, since $\psi$ satisfies properties analogous to \emph{(i)}--\emph{(iii)} by \cite[Theorem 9.1]{GK}, properties \emph{(i)}--\emph{(iii)} for $\phi$ easily follow (see also \cite[Theorem 3.2]{CdRMR}). Then equation \eqref{e:E} has a wavefront with profile $\phi$ and speed \eqref{eq:speed}. The converse implication follows directly.

\smallskip
\noindent \emph{(a2)} \ We denote
\begin{equation}\label{e:GE2}
E(r):=D\left((b-a)r+a\right), \quad  g(r):=\dfrac{f\left((b-a)r+a\right)-f(a)}{b-a}, \quad r\in[0,1].
\end{equation}
As in \emph{(a1)}, we have $E>0$; moreover, $f(\rho)<s(\rho)$ for $\rho \in (a,b)$ if and only if $g(r)/r<g(1)$ for $r\in(0,1)$.  Hence, equation \eqref{e:rtR} with $E$ and $g$ as in \eqref{e:GE2} has a wavefront with profile $\psi \in C^2(I)$ with $I$ as in \eqref{e:defI} satisfying \eqref{e:psi1} if and only if $f(\rho)<s(\rho)$ for $\rho \in (a,b)$. The function $\phi(\xi):=(b-a)\psi(\xi)+a$, $\xi \in \R$, is the profile of a wavefront solution of equation \eqref{e:E} with speed \eqref{eq:speed}; moreover $\phi \in C^2(J)$, it  satisfies \eqref{e:decr} and $\phi^{\, \prime}(\xi)<0$ for $\xi \in J$; properties \emph{(i)}--\emph{(iii)} are deduced as above.

\smallskip

\noindent \emph{(b1)} \ We denote
\begin{equation}\label{e:GE3}
E(r):=-D\left(b-(b-a)r\right), \quad g(r):=\dfrac{f\left(b-(b-a )r\right)-f(b)}{b-a}, \quad r\in[0,1].
\end{equation}
Then $E>0$ in $(0,1)$, $g(0)=0$ and condition $g(r)>g(1)r$, $r\in (0,1)$, is equivalent to $f(\rho)>s(\rho)$, $\rho \in (a,b)$. By \emph{(a1)}, equation \eqref{e:rtR} with $g$ and $E$ as in \eqref{e:GE3} has a wavefront with profile $\psi$ satisfying $\psi(-\infty) =0$, $\psi(+\infty)=1$
if and only if $f(\rho)>s(\rho)$ for $\rho \in (a,b)$.
The speed is
\begin{equation}\label{e:cpsi}
c_{\psi}=\dfrac{f(a)-f(b)}{b-a},
\end{equation}
$\psi \in C^1(I)$ with $I$ as in \eqref{e:defI}, $\psi^{\, \prime}(\xi) >0$ for $\xi \in I$ and $\psi$  satisfies \emph{(i)}--\emph{(iii)}. The function $\phi(\xi):=b-(b-a)\psi(\xi)$, $\xi \in \R$, satisfies \eqref{e:decr} and $\phi \in C^2(J)$; $\phi$ is the  profile of a wavefront solution of \eqref{e:E} with speed \eqref{eq:speed}, $\phi^{\, \prime}(\xi)<0$ for $\xi \in J$, and properties \emph{(i)}--\emph{(iii)} hold true.

\smallskip
\noindent \emph{(b2)} \ We denote once more
\begin{equation}\label{e:GE4}
E(r):=-D\left(b-(b-a)r\right), \qquad g(r):=\dfrac{f\left(b-(b -a)r\right)-f(b)}{b-a}, \qquad r\in[0,1].
\end{equation}
Then $E>0$ in $(0,1)$, $g(0)=0$ and condition $g(r)<g(1)r$, $r\in (0,1)$, is equivalent to $f(\rho)<s(\rho)$, $\rho \in (a,b)$. By \emph{(a2)}, equation \eqref{e:rtR} with $g$ and $E$ as in \eqref{e:GE4} has a wavefront with profile $\psi$ and speed \eqref{e:cpsi} if and only if $f(\rho)<s(\rho)$, $\rho \in (a,b)$; moreover, $\psi \in C^2(I)$ with $I$ as in \eqref{e:defI}, condition \eqref{e:psi1} holds true and $\psi$ satisfies  \emph{(i)}--\emph{(iii)}. The function
$\phi(\xi):=b-(b -a)\psi(\xi)$, $\xi \in \R$, satisfies \eqref{e:incr}, $\phi \in C^2(J)$, $\phi^{\, \prime}(\xi)>0$ for $\xi \in J$, and it is the profile of a wavefront solution of \eqref{e:E} with speed \eqref{eq:speed}; properties \emph{(i)}--\emph{(iii)} hold true.
\end{proof}

\begin{example}Several examples fall in the framework of Theorem \ref{t:wfsConst}; we refer in the following to the notation in {\rm (fcm)} introduced in Section \ref{sec:ex}.

For instance, consider $v(\rho)=1-\rho$; then $f$ is strictly concave. If $D$ is given by \eqref{e:DN}, then $D(\rho) = \rho(\delta-\tau\rho)$ and {\em (a1)} applies in $(0,1)$ if $\frac\delta\tau =v^s\ge1$; notice that \eqref{e:vvs} holds in this case. If $D$ is given by \eqref{e:DT} then $D(\rho)= \delta\rho$ and the same result holds true. Therefore case {\em (a1)} applies in $(0,1)$.

Under the Kladek law $\eqref{e:2v}_1$ we pointed out in Section \ref{sec:ex} that  $f$ is strictly concave; if $D$ is as in \eqref{e:DT}, then $D(0)=0$ and $D(\rho)>0$ if $\rho\in(0,1]$. Case {\em (a1)} applies in $(0,1)$.

At last consider $v(\rho)=\min\left\{1,-c\log\rho\right\}$ for $c>0$, see \cite{Nelson_2000}.
In the interval $I=(e^{-1/c}, 1)$ the function $f$ is strictly concave. If $D$ is given by \eqref{e:DN}, then $D(\rho)= c(\delta-\tau c)$ in $I$ and hence $D\gtrless0$ in $I$ if $\delta/\tau\gtrless c$; condition \eqref{e:vvs} holds if $\delta/\tau=v^s\ge1$. Case {\em (a1)} ({\em (b1)}, respectively) applies in $I$.
\end{example}
\section{Proofs}\label{s:proof}
\setcounter{equation}{0}

\begin{proofof}{Lemma \ref{l:lemma1}}
Assume that \eqref{e:E} has a wavefront solution with profile $\phi$ satisfying \eqref{e:infty}. We already noticed that $\phi$ satisfies \eqref{e:xi0} with no loss of generality. Let $\oxi\ge 0$ be such that $\phi(\oxi)=\alpha$ and $\phi(\xi)>\alpha$ if $\xi>\oxi$. First, we prove
\begin{equation}\label{e:limiti}
\emph{(i)} \quad \lim_{\xi \to 0^-}D\left(\phi(\xi)\right)\phi^{\, \prime}(\xi)=0, \hskip 1.5 cm \emph{(ii)} \quad \lim_{\xi \to \oxi\,^+}D\left(\phi(\xi)\right)\phi^{\, \prime}(\xi)=0.
\end{equation}
The reasoning is slightly different according to $\oxi>0$ or $\oxi=0$; we begin with the case $\oxi>0$.

Consider $h>0$ such that $h<\xi_1$ and $\phi(\xi)>0$ for $\xi \in (-h, h)$; this is possible because $\alpha\in(0,1)$. Let $\psi \in C^{\infty}_0(-h,h)$. Since $\phi$ is a solution of \eqref{e:ODE} (see Definition \ref{d:tws}) we obtain
\begin{align*}
0&=\int_{-h}^{h}\left( D\left(\phi(\xi)\right)\phi^{\, \prime}(\xi) - f\left(\phi(\xi)\right) + c\phi(\xi) \right)\psi'(\xi)\,d\xi\\
&=\displaystyle{\lim_{\epsilon \to 0^+}}\int_{-h}^{-\epsilon}\left( D\left(\phi(\xi)\right)\phi^{\, \prime}(\xi) - f\left(\phi(\xi)\right) + c\phi(\xi) \right)\psi'(\xi)\,d\xi+\left( f(\alpha)-c\alpha\right)\psi(0).
\end{align*}
 The function $D(\phi)\phi^{\, \prime}$ is continuous in every interval $(-h, -\epsilon)$ with $0<\epsilon<h$, and $\phi$ satisfies there equation \eqref{e:ODE}  (see Remark \ref{r:ac}); hence
\begin{align*}
0  =& \displaystyle{\lim_{\epsilon \to 0^+}}\int_{-h}^{-\epsilon}\left( D\left(\phi(\xi)\right)\phi^{\, \prime}(\xi) - f\left(\phi(\xi)\right) + c\phi(\xi) \right)\psi'(\xi)\,d\xi+\left( f(\alpha)-c\alpha\right)\psi(0)
\\
= & \displaystyle{\lim_{\epsilon \to 0^+}} \left( D\left(\phi(-\epsilon)\right)\phi^{\, \prime}(-\epsilon)- f\left(\phi(-\epsilon )\right) + c\phi(-\epsilon) \right)\psi(-\epsilon)
+\left( f(\alpha)-c\alpha\right)\psi(0).
\end{align*}
Since we may assume $\psi(0)\ne 0$, by the continuity of $f$ and $\phi$ we obtain condition \eqref{e:limiti}\emph{(i)}.
Similarly, one can prove \eqref{e:limiti}\emph{(ii)} when $\oxi>0$.

Now, we consider the case $\oxi=0$. Choose $h>0$ such that $0<\phi(\xi)<1$ for $\xi \in (-h, h)$ and
take $\psi \in C^{\infty}_0(-h, h)$. Again by Definition \ref{d:tws} we obtain
\begin{align*}
0=&\int_{-h}^{h}\left( D\left(\phi(\xi)\right)\phi^{\, \prime}(\xi) - f\left(\phi(\xi)\right) + c\phi(\xi) \right)\psi'(\xi)\,d\xi\\
=&\displaystyle{\lim_{\epsilon \to 0^+}}\int_{-h}^{-\epsilon}\left( D\left(\phi(\xi)\right)\phi^{\, \prime}(\xi) - f\left(\phi(\xi)\right) + c\phi(\xi) \right)\psi'(\xi)\,d\xi\\
&+\displaystyle{\lim_{\epsilon \to 0^+}}\int_{\epsilon}^{h}\left( D\left(\phi(\xi)\right)\phi^{\, \prime}(\xi) - f\left(\phi(\xi)\right) + c\phi(\xi) \right)\psi'(\xi)\,d\xi.
\end{align*}
For $0<\epsilon<h$, by the regularity of $\phi$ in both $(-h, -\epsilon)$ and $(\epsilon, h)$ (see Remark \ref{r:ac}) we obtain
\begin{align*}
0 = &\displaystyle{\lim_{\epsilon \to 0^+}}\int_{-h}^{-\epsilon}\left( D\left(\phi(\xi)\right)\phi^{\, \prime}(\xi) - f\left(\phi(\xi)\right) + c\phi(\xi) \right)\psi'(\xi)\,d\xi
\\
&+\displaystyle{\lim_{\epsilon \to 0^+}}\int_{\epsilon}^{h}\left( D\left(\phi(\xi)\right)\phi^{\, \prime}(\xi) - f\left(\phi(\xi)\right) + c\phi(\xi) \right)\psi'(\xi)\,d\xi\\
=&\displaystyle{\lim_{\epsilon \to 0^+}} \left( D\left(\phi(-\epsilon)\right)\phi^{\, \prime}(-\epsilon)- f\left(\phi(-\epsilon )\right) + c\phi(-\epsilon) \right)\psi(-\epsilon)
\\
&-\displaystyle{\lim_{\epsilon \to 0^+}} \left( D\left(\phi(\epsilon)\right)\phi^{\, \prime}(\epsilon)- f\left(\phi(\epsilon )\right) + c\phi(\epsilon) \right)\psi(\epsilon)
\\
= &\displaystyle{\lim_{\epsilon \to 0^+}} \left( D\left(\phi(-\epsilon)\right)\phi^{\, \prime}(-\epsilon)\right)\psi(-\epsilon)  - \displaystyle{\lim_{\epsilon \to 0^+}}\left( D\left(\phi(\epsilon)\right)\phi^{\, \prime}(\epsilon)\right)\psi(\epsilon).
\end{align*}
The above expression must be satisfied in particular when $\psi(0)\ne 0$. Notice, moreover, that  $\phi^{\, \prime}\ge 0$ in $(-h,h)\setminus\{0\}$ and $D$ changes sign in $\alpha$; hence both \eqref{e:limiti}\emph{(i)} and
 \eqref{e:limiti}\emph{(ii)} are satisfied. This completely proves \eqref{e:limiti}.

\smallskip

Let $\phi$ be as in the first part of this proof and take $\xi_1>0$; the case $\xi_1=0$ is proved analogously. We claim that  the function
\begin{equation*}
\phi_1(\xi)=\left\{
\begin{array}{rl}
\phi(\xi) \quad &\hbox{ if }\xi\le 0,\\
\alpha \quad &\hbox{ if }\xi \in (0, \xi_1),\\
\phi(\xi+\oxi-\xi_1) \quad &\hbox{ if }\xi\ge \xi_1,
\end{array}
\right.
\end{equation*}
is the profile of a wavefront to \eqref{e:E}.

As above, fix $h \in (0, \xi_1)$ such that $\phi(\xi)>0$ for $\xi \in (-h, 0)$ and consider $\psi \in C_0^{\infty}(-h, h)$. By condition \eqref{e:limiti}\emph{(i)} we obtain
\begin{align*}
&\ds\int_{-h}^{h}\left(D\left(\phi_1(\xi)\right)\phi_1'(\xi) - f\left(\phi_1(\xi)\right) + c\phi_1(\xi) \right)\psi'(\xi)\,d\xi
\\
=&\ds\int_{-h}^{0}\left(D\left(\phi(\xi)\right)\phi'(\xi) - f\left(\phi(\xi)\right) + c\phi(\xi) \right)\psi'(\xi)\,d\xi +\left(f(\alpha) - c\alpha\right)\psi(0)\\
=&\displaystyle{\lim_{\epsilon \to 0^+}}\int_{-h}^{-\epsilon}\left( D\left(\phi(\xi)\right)\phi^{\, \prime}(\xi) - f\left(\phi(\xi)\right) + c\phi(\xi) \right)\psi'(\xi)\,d\xi+\left(f(\alpha) - c\alpha\right)\psi(0)\\
=&\displaystyle{\lim_{\epsilon \to 0^+}} \left( D\left(\phi(-\epsilon)\right)\phi^{\, \prime}(-\epsilon)- f\left(\phi(-\epsilon )\right) + c\phi(-\epsilon) \right)\psi(-\epsilon)+\left(f(\alpha) - c\alpha\right)\psi(0)
=\,0.
\end{align*}
Similarly, by condition \eqref{e:limiti}\emph{(ii)} we obtain
\begin{align*}
&\ds\int_{\xi_1-h}^{\xi_1+h}\left(D\left(\phi_1(\xi)\right)\phi_1'(\xi) - f\left(\phi_1(\xi)\right) + c\phi_1(\xi) \right)\psi'(\xi)\,d\xi=0,
\end{align*}
for $h\in (0, \xi_1)$  with $\phi(\oxi+h)<1$ and $\psi\in C_0^{\infty}(\xi_1-h, \xi_1+h)$.  Hence, $\phi_1$ is the profile of a wavefront of \eqref{e:E} with $\phi_1(\pm\infty)=\ell^\pm$.
\end{proofof}

\begin{proofof}{Theorem \ref{t:main}} Assume that there exists a profile $\phi$ satisfying \eqref{e:infty} with speed $c$; by Lemma \ref{l:lemma1}, without loss of generality we may assume that it satisfies \eqref{e:xi0xi1=0}. We have to show that $c=c_{\ell^{\pm}}$, i.e., $c$ satisfies \eqref{e:existence wfs}, and that conditions \eqref{e:existence about f} hold true. Consider
\begin{equation*}
\phi_1(\xi)=\left\{
\begin{array}{rl}
\phi(\xi) \quad &\hbox{ if }\xi\le 0,\\
\alpha \quad &\hbox{ if }\xi>0.
\end{array}
\right.
\end{equation*}
With a reasoning as in the proof of Lemma \ref{l:lemma1}, by condition \eqref{e:limiti}\emph{(i)} we deduce that $\phi_1$ is the profile of a wavefront to \eqref{e:E}.
By Theorem \ref{t:wfsConst}\emph{(a1)} the speed of $\phi_1$ is
$$
c_1:=\frac{f(\alpha)-f(\ell^-)}{\alpha-\ell^-},
$$
and condition \eqref{e:existence about f}$_1$ is satisfied; then, $c=c_1$. Similarly, the function
\begin{equation*}
\phi_2(\xi)=\left\{
\begin{array}{rl}
\alpha \quad &\hbox{ if }\xi<0,\\
\phi(\xi) \quad &\hbox{ if }\xi\ge 0,
\end{array}
\right.
\end{equation*}
is the profile of a wavefront of \eqref{e:E} with $\phi_2(-\infty)=\alpha$ and $\phi_2(\infty)= \ell^+$; hence,  by Theorem \eqref{t:wfsConst}\emph{(b2)}, its speed is
$$
c_2:=\frac{f(\ell^+)-f(\alpha)}{\ell^+-\alpha}
$$
and condition \eqref{e:existence about f}$_2$ is satisfied; moreover, $c=c_2$. We deduce $c=c_1=c_2$ and so condition \eqref{e:existence wfs} is satisfied. Therefore we proved that, if $\phi$ exists, then \eqref{e:existence wfs} and \eqref{e:existence about f} are satisfied.

\smallskip
Conversely, assume that \eqref{e:existence wfs} and \eqref{e:existence about f}  are satisfied. Equation \eqref{e:E} with $\rho \in [\ell^-, \alpha]$ satisfies assumption {\em (a1)} of Theorem \ref{t:wfsConst}. Hence, it has a wavefront solution with profile $\phi^-$ from $\ell^-$ to $\alpha$, speed $c^-$ and $(\phi^-)'(\xi)>0$ for $\ell^-<\phi^-(\xi)<\alpha$. Moreover, $\phi^-$ is sharp in $\alpha$, since $D(\alpha)=0$; it is also sharp in $\ell^-$ if and only if $\ell^-=0= D(0)$. Let $\xi_0\in \R$ be such that $\phi^-(\xi_0)=\alpha$ and $\phi^-(\xi)<\alpha$ for $\xi<\xi_0$; we may assume $\xi_0=0$. By Theorem \ref{t:wfsConst}\emph{(ii)} we deduce
\begin{equation}\label{e:C1}
\lim_{\xi \to 0^-}D\left(\phi^-(\xi)\right)(\phi^-)'(\xi)=0.
\end{equation}

\smallskip
\noindent  Equation \eqref{e:E} with $\rho \in [\alpha, \ell^+]$ satisfies assumption {\em (b2)} of Theorem \ref{t:wfsConst}; therefore, it has a wavefront solution with profile $\phi^+$ from $\alpha$ to $\ell^+$, speed $c^+$ and $(\phi^+)'(\xi)>0$ for $\alpha<\phi^+(\xi)$; moreover $\phi^+$ is sharp at $\alpha$. As above, we can assume $\phi^+(0)=\alpha$ with $\phi^+(\xi)>\alpha$ for $\xi>0$. Again by Theorem \ref{t:wfsConst}\emph{(i)} we have
\begin{equation}\label{e:C2}
\lim_{\xi \to 0^+}D\left(\phi^+(\xi)\right)(\phi^+)'(\xi)=0.
\end{equation}

We consider now the function
\begin{equation*}
\phi(\xi):=\left\{
\begin{array}{rl}
\phi^-(\xi) & \hbox{ if }\xi\in (-\infty, 0],\\
\phi^+(\xi) & \hbox{ if }\xi\in (0, +\infty).
\end{array}\right.
\end{equation*}
By the properties of $\phi^\pm$ and conditions \eqref{e:C1}-\eqref{e:C2}, we have that $\phi$ is a solution of \eqref{e:ODE} on the whole real line (see Definition \ref{d:tws} and the reasoning in the first part of this proof) with same $c=c_{\ell^{\pm}}$. Hence $\phi$ is the profile of a wavefront of \eqref{e:E} and satisfies all the required properties (the sharpness of $\phi$ has been discussed above, the strict monotonicity follows as well); in particular its speed is $c_{\ell^{\pm}}$.

\smallskip

Now, we prove \eqref{e:slope}. By \eqref{e:existence about f} we have
\[
\frac{f(\rho)-f(\alpha)}{\rho-\alpha}<\frac{f(\alpha)-f(\ell^-)}{\alpha-\ell^-},\quad \hbox{ for } \rho\in (\ell^-,\alpha).
\]
Then, we obtain \eqref{e:slope} by passing to the limit when $\rho \to \alpha^-$.

\smallskip

At last, we prove \eqref{e:slopeLR}. Let
\begin{equation*}
\xi_{\ell^-}:=\inf\left\{\xi\in \mathbb{R} \, : \, \phi(\xi)>\ell^-\right\}.
 \end{equation*}
Notice that $\xi_{\ell^-}\in \mathbb{R}$ if and only if $\ell^-=0=D(0)$, while $\xi_{\ell^-}=-\infty$ otherwise. Since $\phi$ coincides with the profile $\phi^-$ defined above when $\xi<0$, by Theorem \ref{t:wfsConst} {\em (i), (iii)} we obtain
\begin{equation}\label{e:in-}
\lim_{\xi \to \xi_{\ell^-} }D\left(\phi(\xi)\right)\phi^{\, \prime}(\xi)=0.
\end{equation}
 By integrating equation \eqref{e:ODE} in an interval $[\xi_0, \xi]\subset (\xi_{\ell^-} , 0)$, we have, by Remark \ref{r:ac},
\begin{equation*}
D\left(\phi(\xi)\right)\phi^{\, \prime}(\xi)-D\left(\phi(\xi_0)\right)\phi^{\, \prime}(\xi_0)+c_{\ell_{\pm}}\left( \phi(\xi)-\phi(\xi_0)\right)-f\left(\phi(\xi)\right)+f\left(\phi(\xi_0)\right)=0.
\end{equation*}
Hence, by passing to the limit when $\xi_0\to \xi_{\ell^-} $, by \eqref{e:in-} we obtain
\begin{equation}\label{e:phid}
\phi^{\, \prime}(\xi)=\frac{f\left(\phi(\xi)\right)-f(\ell^-)-c_{\ell_{\pm}}\left( \phi(\xi)-\ell^-\right)}{D\left(\phi(\xi)\right)}, \quad \xi \in (\xi_{\ell^-}, 0).
\end{equation}
In particular this shows that $\phi'(\xi)>0$ if $\xi \in (\xi_{\ell^-}, 0)$. By de L'Hospital rule, we obtain
\begin{align}\label{e:HH}
\displaystyle{\lim_{\xi \to 0^-}}\phi^{\, \prime}(\xi)&=\displaystyle{\lim_{\sigma \to \alpha^-}}\frac{f(\sigma)-f(\ell^-)-c_{\ell_{\pm}}\left(\sigma-\ell^-\right)}{D(\sigma)}
=\displaystyle{\lim_{\sigma \to \alpha^-}}\frac{f^{\, \prime}(\sigma)-c_{\ell_{\pm}}}{D^{\, \prime}(\sigma)},
\end{align}
whence $\phi^{\prime}_-(0)$ satisfies \eqref{e:slopeLR}. Since the reasoning for $\phi^{\prime}_+(0)$ is similar, this completely proves \eqref{e:slopeLR}. Notice that by \eqref{e:HH} we immediately deduce \eqref{e:00}. The proof is complete.
\end{proofof}

\begin{example}
The proof of Theorem \ref{t:main} shows how to deal with similar cases. Consider for instance the velocity in $\eqref{e:2v}_2$ with $0<a<1$, $f$ as in {\rm (fcm)} in Section \ref{sec:ex}, $D$ given by \eqref{e:DN}, and assume \eqref{e:vvs}. Here, it is convenient to use the notation $v^s=\delta/\tau$ and $w:=\ov/v^s\in(0,1]$. First, we claim that if
\begin{equation}\label{e:PD0}
\gamma >\frac{1-a}{aw}\quad \hbox{ and } \quad  \gamma \ge\frac{1+a}{a},
\end{equation}
then there is $\beta \in (a,1)$ such that  $D<0$ in $(a, \beta)$ and $D>0$ in $(\beta, 1)$.

To prove the claim, first notice that $\eqref{e:PD0}_2$ implies $\gamma>2$ because $0< a<1$. Moreover,
\begin{equation}\label{e:DLu}
D(\rho) = \frac{\tau\overline v^2 \gamma^2 (1-a)^2}{(1-\rho)^4}\rho^2 e^{\gamma\frac{a-\rho}{1-\rho}}h(\rho), \quad \rho\in (a,1),
\end{equation}
for
\[
h(\rho):=\frac{1}{\gamma w (1-a)}\frac{(1-\rho)^2}{\rho}- e^{\gamma\frac{a-\rho}{1-\rho}},\quad \rho\in (a,1).
\]
We pointed out in Section \ref{sec:ex} that $h(\rho)>0$ in a left neighborhood of $1$; moreover
$\eqref{e:PD0}_1$ is equivalent to $h(a)<0$. Then, there is  $\beta \in (a,1)$ such that $h(\beta)=0$, implying
$$
\frac{1}{\gamma w(1-a)}\frac{(1-\beta)^2}{\beta}=e^{\gamma\frac{a-\beta}{1-\beta}}.
$$
We compute $h'(\beta)=\frac{1}{ \gamma w(1-a)\beta^2}\,\psi(\beta)$, for
$\psi(\rho):=\rho^2 +\gamma(1-a)\rho-1$ and $\rho \in (a,1)$.
By $\eqref{e:PD0}_2$ we deduce $\psi(a) = (1-\gamma)a^2+\gamma a -1\ge 0$. Moreover $\psi'(\rho)>0$ for $\rho>0$, and so $\psi(\beta)>0$ whatever $\beta \in (a,1)$ is. Then $h'(\beta)>0$ and so $\beta$ is unique. This proves our claim.

Notice that the sign of $D$ in \eqref{e:DLu} is {\em opposite} to that in condition {\rm (D1)}, so Theorem \ref{t:main} cannot apply as it is. However, we can \lq\lq paste\rq\rq\ the profiles of cases (b1) and (a2) in Theorem \ref{t:wfsConst} as in the proof of Theorem \ref{t:main}, and obtain a {\em decreasing} profile. More precisely, assume $a\gamma <2$; it is always possible to choose $a$ satisfying this condition and $\eqref{e:PD0}_2$ because $a<1$. Then $f(\rho)$ is concave in $(a, \tilde\rho)$ and convex in $(\tilde\rho,1)$, see the discussion preceding \eqref{e:tilderho}. If conditions \eqref{e:existence about f} are satisfied with $\beta$ replacing $\alpha$ and $a\le\ell^-<\beta <\ell^+\le 1$, then equation \eqref{e:E} admits a wavefront solution with profile $\varphi$ satisfying $\varphi(-\infty)=\ell^+$ and $\varphi(-\infty)=\ell^-$.
\end{example}
In the proof of Theorem \ref{t:conv left} we shall need the following result. We previously observe that a byproduct of the proof of Lemma \ref{l:lemma1} is formula \eqref{e:phid}; for profiles $\phi_\eps$ of traveling-wave solutions of \eqref{e:Eepsilon} it can be written as
\begin{equation}\label{e:deriv}
\phi^{\, \prime}_{\epsilon}(\xi)=\frac{f\left(\phi_{\epsilon}(\xi)\right)-f(\ell^-)-c_{\ell_{\pm}}\left( \phi_{\epsilon}(\xi)-\ell^-\right)}{{\epsilon}D\left(\phi_\epsilon(\xi)\right)},
\end{equation}
for every $\xi$ such that $0<\phi_{\epsilon}(\xi)<1$ and $\phi_{\epsilon}(\xi)\ne\alpha$.

\begin{lemma}\label{l:forTHCL}
Under the assumptions of Theorem \ref{t:conv left} we have
\begin{equation}\label{e:eps1eps2}
\phi_{\epsilon_1}(\xi)\le \phi_{\epsilon_2}(\xi),\quad \hbox{ for every }\quad  \xi <0,\ \epsilon_1< \epsilon_2.
\end{equation}
Moreover, we have $\phi_{\epsilon_1}(\xi_0)= \phi_{\epsilon_2}(\xi_0)$ for some $\xi_0<0$ if and only if $\phi_{\epsilon_1}(\xi)=\phi_{\epsilon_2}(\xi)=0$ for every $\xi<\xi_0$.
\end{lemma}
\begin{proof}
By \eqref{e:slope} we know that $f^{\, \prime}(\alpha)\le c_{\ell_{\pm}}$. For clarity, we split the proof of \eqref{e:eps1eps2} into two parts, and each of them is further split. At the end we shall prove the last claim of the statement.

\noindent\emph{(i) Assume $f^{\, \prime}(\alpha)<c_{\ell^{\pm}}$}. First, we show that \eqref{e:eps1eps2} holds in a left neighborhood of $0$ and, then, in the whole interval $(-\infty,0]$.

{\em (a)}\ By \eqref{e:slopeLR}  we have
\begin{equation}\label{e:phiineq}
\phi_{\epsilon_1}^{\, \prime}(0)=\frac{f^{\, \prime}(\alpha)-c_{\ell_{\pm}}}{\epsilon_1D^{\, \prime}(\alpha)}>\frac{f^{\, \prime}(\alpha)-c_{\ell_{\pm}}}{\epsilon_2D^{\, \prime}(\alpha)}=\phi_{\epsilon_2}^{\, \prime}(0).
\end{equation}
Notice that {\em both} conditions $D^{\, \prime}(\alpha)<0$ in \eqref{e:Dslope} and $f^{\, \prime}(\alpha)<c_{\ell^{\pm}}$ are needed to deduce \eqref{e:phiineq}. By $\phi_{\epsilon_1}(0)=\phi_{\epsilon_2}(0)=\alpha$, see \eqref{e:xi0}, we have $\phi_{\epsilon_1}(\xi)<\phi_{\epsilon_2}(\xi)$ at least for $\xi \in (\tilde\xi, 0)$, for some $\tilde\xi<0$.

\smallskip

{\em (b)}\  Assume by  contradiction that there exists $\oxi<\tilde\xi$ such that $\phi_{\epsilon_1}(\oxi)=\phi_{\epsilon_2}(\oxi)=:\eta>0$ while
\begin{equation}\label{e:est1}
\phi_{\epsilon_1}(\xi)<\phi_{\epsilon_2}(\xi)\quad \hbox{ if }\quad \xi \in (\overline{\xi}, 0).
\end{equation}
By \eqref{e:phid}, \eqref{e:existence wfs}, \eqref{e:existence about f}$_1$, $D(\eta)>0$ and $\eps_1<\eps_2$ we deduce
\begin{align*}
\phi_{\epsilon_1}^{\, \prime}(\overline{\xi}) = \frac{f(\eta)-\left( f(\alpha)+c_{\ell_{\pm}}(\eta-\alpha\right)}{\epsilon_1D(\eta)}
>\frac{f(\eta)-\left( f(\alpha)+c_{\ell_{\pm}}(\eta-\alpha\right)}{\epsilon_2D(\eta)} = \phi_{\epsilon_2}^{\, \prime}(\oxi),
\end{align*}
in contradiction with \eqref{e:est1}. This proves claim \eqref{e:eps1eps2} if $f^{\, \prime}(\alpha)<c_{\ell^{\pm}}$.

\smallskip

\noindent\emph{(ii) Assume $f^{\prime}(\alpha)=c_{\ell^{\pm}}$}. In this case, the above deduction of \eqref{e:phiineq} fails. Now, the proof of inequality (5.9) is based on the remark that the profile $\phi_{\epsilon_1}$ can be obtained as the limit of a sequence of suitably shifted functions $\phi_n$ of $\phi_{\epsilon_1}$; a further shift $\phi_0$ of $\phi_{\epsilon_1}$ is introduced to have a uniform bound from below. As in the proof of case {\em (i)}, the first items below concern the proof of \eqref{e:eps1eps2} in a left neighborhood of $0$. We refer to Figure \ref{f:LC}. 

\smallskip

{\em (a)}\ To avoid the possible degeneracy occurring if $D(0)=\ell^-=0$, fix $\sigma \in (\ell^-, \alpha)$ and denote $\phi_0(\xi):=\phi_{\epsilon_1}(\xi+\xi_0)$, where $\xi_0<0$ is such that $\phi_0(0)=\sigma$. Then, fix $\mu \in (\ell^-, \sigma)$ and let $\tau_0<0$ satisfy $\phi_0(\tau_0)=\mu$. Notice that both values $\xi_0$ and $\tau_0$ exist by Theorem \ref{t:main}, and $\phi_0^{\, \prime}(\xi)>0$ for $\xi \in [\tau_0, 0]$. By $\sigma=\phi_0(0)<\phi_{\epsilon_2}(0)=\alpha$, see \eqref{e:xi0}, and arguing as in {\em (i)(b)}, we deduce
 \begin{equation}\label{e:ccll}
 \phi_0(\xi)<\phi_{\epsilon_2}(\xi) \qquad \hbox{ for }\xi \in [\tau_0, 0].
 \end{equation}

\smallskip

{\em (b)}\ Consider a strictly increasing sequence $\{\xi_n\}\subset (\tau_0, 0)$ satisfying $\xi_n \to 0$ for $n \to \infty$ and let $\eta_n:=\phi_{\epsilon_2}(\xi_n), \, n \in \mathbb{N}$. Also the sequence $\{\eta_n\}\subset(\ell^-,\alpha)$ is strictly increasing, by Theorem \ref{t:main}. For $n \in \mathbb{N}$, let  $\phi_n$ be the solution of the initial-value problem
\begin{equation}\label{e:ephin}
\left\{
\begin{array}{l}
\phi^{\, \prime}(\xi) =\frac{f\left(\phi(\xi)\right)-f(\ell^-)-c_{\ell_{\pm}}\left( \phi(\xi)-\ell^-\right)}{{\epsilon_1}D\left(\phi(\xi)\right)},\\
\phi(\xi_n)=\eta_n.
\end{array}
\right.
\end{equation}
Clearly, we have $\phi_n(\xi)=\phi_{\eps_1}(\xi+\zeta_n)$, for suitable shifts $\zeta_n$. Reasoning as in {\em (ii)(a)} we have that $\phi_n$ is defined and strictly increasing in an interval $[\tau_n, 0]$, with $\phi_n(\tau_n)=\mu$. Notice that $\phi_0(\xi_n)<\phi_n(\xi_n) = \eta_n$ for every $n$ by \eqref{e:ccll}; hence, by the uniqueness of the solution of the equation in \eqref{e:ephin},  we deduce $\tau_n <\tau_0$ and
\begin{equation}\label{e:0lessn}
\phi_0(\xi)<\phi_n(\xi) \qquad \hbox{ for }\xi \in [\tau_0, 0] \hbox{ and } \, n \in \mathbb{N}.
\end{equation}

\begin{figure}[htbp]
\begin{picture}(100,135)(80,-15)
\setlength{\unitlength}{1.2pt}

\put(340,0){
\put(0,0){\vector(1,0){30}}
\put(0,0){\line(-1,0){200}}
\put(30,8){\makebox(0,0){$\xi$}}
\put(0,0){\vector(0,1){110}}
\put(-3,110){\makebox(0,0)[r]{$\phi$}}

\put(-2,8){\line(1,0){4}}
\put(4,9){\makebox(0,0)[lt]{$\ell^-$}}
\put(-200,8){\line(1,0){80}}

\put(-2,102){\line(1,0){4}}
\put(4,100){\makebox(0,0)[lb]{$\ell^+$}}

\put(4,80){\makebox(0,0)[l]{$\alpha$}}
\put(4,30){\makebox(0,0)[l]{$\sigma$}}

\put(0,0){\thicklines{\qbezier(-200,15)(-60,40)(0,80)}}
\put(-140,32){\makebox(0,0)[b]{$\phi_{\eps_2}$}}

\put(0,0){\thicklines{\qbezier(-200,10)(-60,20)(0,30)}}
\put(-20,24){\makebox(0,0)[t]{$\phi_0$}}

\put(0,0){\qbezier(-90,25)(-40,38)(0,80)}
\put(-50,37){\makebox(0,0)[t]{$\phi_{\eps_1}$}}

\put(0,0){\thicklines{\qbezier(-100,25)(-50,40)(-10,90)}}
\put(-12,92){\makebox(0,0)[b]{$\phi_{n+1}$}}

\put(0,0){\thicklines{\qbezier(-120,25)(-70,40)(-30,90)}}
\put(-32,92){\makebox(0,0)[b]{$\phi_{n}$}}

\multiput(-41,0)(0,5){12}{$.$}
\put(-30,-3){\makebox(0,0)[t]{$\xi_{n+1}$}}
\multiput(-1,58)(-5,0){9}{$.$}
\put(4,58){\makebox(0,0)[l]{$\eta_{n+1}$}}

\multiput(-81,0)(0,5){9}{$.$}
\put(-81,-3){\makebox(0,0)[t]{$\xi_{n}$}}
\multiput(-1,43)(-5,0){17}{$.$}
\put(4,43){\makebox(0,0)[l]{$\eta_{n}$}}

\multiput(-201,0)(0,5){2}{$.$}
\put(-201,-3){\makebox(0,0)[t]{$\tau_0$}}
\multiput(-1,9)(-5,0){40}{$.$}
\put(4,10){\makebox(0,0)[lb]{$\mu$}}

}

\end{picture}
\caption{\label{f:LC}{For the proof of Lemma \ref{l:forTHCL}.}}
\end{figure}
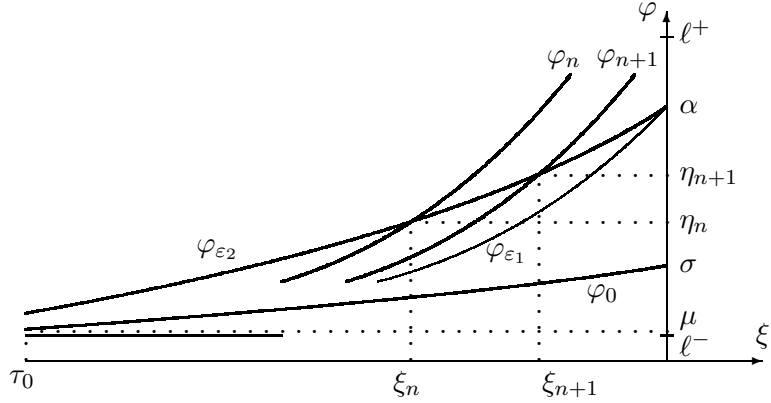

\smallskip

{\em (c)} We claim that
\begin{equation}\label{e:claim1n}
\phi_n(\xi)<\phi_{\epsilon_2}(\xi) \qquad \hbox{ for }\xi \in [\tau_0, \xi_n) \hbox{ and } \, n \in \mathbb{N}.
\end{equation}
In fact, as in case {\em (i)(a)} with $\xi_n$ replacing $\overline\xi$, by \eqref{e:existence about f}$_1$, \eqref{e:deriv} and $D(\eta_n)>0$ we obtain
\begin{equation*}
0<\phi_{\epsilon_2}^{\, \prime}(\xi_n)=\frac{f\left(\eta_n\right)-f(\ell^-)-c_{\ell_{\pm}}\left( \eta_n-\ell^-\right)}{{\epsilon_2}D\left(\eta_n\right)}<\frac{f\left(\eta_n\right)-f(\ell^-)-c_{\ell_{\pm}}\left( \eta_n-\ell^-\right)}{{\epsilon_1}D\left(\eta_n\right)}=\phi_n^{\, \prime}(\xi_n).
\end{equation*}
Then $\phi_n(\xi)<\phi_{\epsilon_2}(\xi)$ in a left neighborhood of $\xi_n$; one shows \eqref{e:claim1n} by arguing as in case {\em (i)(b)}.

\smallskip

{\em (d)} We denote
\begin{equation*}
\hat \phi_n(\xi)=\left\{
\begin{array}{rl}
\phi_n(\xi) & \text{ if } \tau_0<\xi\le \xi_n,
\\[2mm]
\phi_{\epsilon_2}(\xi) & \text{ if }  \xi_n<\xi\le 0,
\end{array}
\right.
\quad \hbox{ and } \quad \hat \phi(\xi)=\inf_{n \in \mathbb{N}} \hat\phi_n(\xi), \qquad \hbox{ for }\xi \in [\tau_0, 0].
\end{equation*}
We claim that the sequence $\{\hat \phi_n\}$ is decreasing in $[\tau_0, 0]$. In the interval $[\xi_{n+1},0]$ we have $\phi_{n+1}=\phi_n=\phi_{\eps_2}$. In $(\xi_n, \xi_{n+1})$ we have $\phi_{n+1} <\phi_n=\phi_{\eps_2}$ by \eqref{e:claim1n}. In $[\tau_0,\xi_n)$ we still have $\phi_{n+1} <\phi_n$ because both of them are shifts of the same profile and are strictly increasing in $[\tau_0,0]$ by {\em (ii)(b)}. Then $\{\hat \phi_n\}$ is decreasing in $[\tau_0, 0]$ and by \eqref{e:ccll}, \eqref{e:0lessn} we have
\begin{equation*}
\phi_0(\xi)\le \hat \phi(\xi)= \displaystyle{\lim_{n\to\infty}}\hat\phi_n(\xi) \qquad \hbox{ for }\xi \in [\tau_0, 0].
\end{equation*}

\smallskip

{\em (e)} Now we prove
\begin{equation}\label{e:phiphi}
\hat \phi=\phi_{\epsilon_1}\quad \hbox{ in } [\tau_0, 0].
\end{equation}
Equation \eqref{e:phiphi} is obviously satisfied in $\xi=0$. Notice that by \eqref{e:existence wfs} and \eqref{e:existence about f}$_1$ we have
 \begin{equation*}
 \psi(\rho):=\frac{f(\rho)-f(\ell^-)-c_{\ell_{\pm}}\left( \rho-\ell^-\right)}{{\epsilon_1}D(\rho)}>0 \qquad \hbox{ for }\rho \in(\ell^-, \alpha).
 \end{equation*} 
Consider an interval $[a, b]\subset [\tau_0, 0)$; since $\psi\in C^1([a,b])$, we deduce that $\{\hat\phi_n\}$ is equicontinuous and then relatively compact in $[a,b]$ by Ascoli-Arzel\`{a} Theorem. We can then extract a subsequence, which is denoted as usual as the whole sequence, such that $\hat \phi_n\to\hat \phi$ uniformly in $[a,b]$. This implies that also the sequence
\begin{equation*}
\left\{\frac{f\left(\hat \phi_n(\xi)\right)-f(\ell^-)-c_{\ell_{\pm}}\left( \hat \phi_n(\xi)-\ell^-\right)}{{\epsilon_1}D\left(\hat \phi_n(\xi)\right)}\right\}
\end{equation*}
is uniformly convergent  in $[a,b]$. Choose $N$ large enough in such a way that $\xi_n >b$ for $n\ge N$; then $\hat \phi_n = \phi_n$ in $[a,b]$ for $n\ge N$. By passing to the limit in the identity
 \begin{equation*}
\hat \phi_n(b)-\hat \phi_n(a) =\int_a^b\frac{f\left(\hat\phi_n(\xi)\right)-f(\ell^-)-c_{\ell_{\pm}}\left( \hat\phi_n(\xi)-\ell^-\right)}{{\epsilon_1}D\left(\hat\phi_n(\xi)\right)}\, d\xi,\qquad n\ge N,
\end{equation*} 
we obtain
\begin{equation*}
\hat \phi(b)-\hat \phi(a) =\int_a^b\frac{f\left(\hat \phi(\xi)\right)-f(\ell^-)-c_{\ell_{\pm}}\left( \hat \phi(\xi-\ell^-\right)}{{\epsilon_1}D\left(\hat \phi(\xi)\right)}\, d\xi.
\end{equation*}
This proves \eqref{e:phiphi}.

{\em (f)} By \eqref{e:claim1n}, condition \eqref{e:eps1eps2} is satisfied in $[\tau_0, 0]$; we prove \eqref{e:eps1eps2} for $\xi <\tau_0$ with a reasoning as in {\em (i)(b)}. 

\smallskip

This concludes the proof of \eqref{e:eps1eps2}. We are left with the last claim of the statement in the lemma. If either $D(0)>0$ or $\ell^->0$, an argument as in {\em (i)(b)} shows that \eqref{e:eps1eps2} is satisfied with the strict inequality. When  $D(0)=\ell^-=0$ both profiles are sharp by Theorem \ref{t:main}; as a consequence, there exists $\xi_0<0$ such that $\phi_{\epsilon_1}(\xi)=\phi_{\epsilon_2}(\xi)=0$ for $\xi\le \xi_0$.
\end{proof}

\begin{remark}
With reference to the proof of Lemma \ref{l:forTHCL}, by \eqref{e:claim1n} we have $\phi_{\eps_1}(\xi_n+\zeta_{n+1})=\phi_{n+1}(\xi_n)< \phi_{\eps_2}(\xi_n)= \phi_n(\xi_n)=\phi_{\eps_1}(\xi_n+\zeta_n)$ and so $\zeta_{n+1}<\zeta_n$. Moreover, by arguing as to prove \eqref{e:claim1n}, one shows $\phi_n>\phi_{\epsilon_2}$ in $(\xi_n,0]$, whence $\phi_{\eps_1}(\zeta_n)=\phi_n(0)>\phi_{\eps_2}(0)=\alpha$. All in all, $0<\zeta_{n+1}<\zeta_n$.
\end{remark}

\begin{proofof}{Theorem \ref{t:conv left}}
By Lemma \ref{l:forTHCL} the family $\{\phi_\eps\}_{\eps}$ is decreasing in $(-\infty,0]$ by \eqref{e:eps1eps2}, we can define
\begin{equation}\label{e:phi0}
\phi_0(\xi):=\lim_{\epsilon \to 0^+}\phi_{\epsilon}(\xi), \qquad \xi \le 0.
\end{equation}
We have $\phi_0(\xi)\ge \ell^-$ for every $\xi \in(-\infty,0]$; moreover, $\phi_0$ is monotone increasing in $(-\infty,0]$ as pointwise limit of increasing functions. As a consequence, if $\phi_0(\overline{\xi})>\ell^-$ for some $\overline{\xi}<0$ then $\phi_0>\ell^-$ in the interval $[\overline{\xi}, 0]$.

We prove now \eqref{e:conv left}. We reason by contradiction and assume the existence of $\xi_0<0$ such that $\phi_0(\xi_0)>\ell^-$. By what we have noticed just above, we deduce $\phi_0(\xi)>\ell^-$ for $\xi \in [\xi_0, 0]$. Let $\xi_1\in (\xi_0, 0)$. By the monotonicity of each $\phi_{\epsilon}, \, \epsilon \in [0,1]$, we have
\begin{equation*}
\ell^-<\mu_0:=\phi_0(\xi_0)\le\phi_0(\xi)< \phi_{\epsilon}(\xi)<\phi_1(\xi)\le\phi_1(\xi_1)=:\mu_1<\alpha,
\end{equation*}
for $\xi \in [\xi_0, \xi_1]$, $\epsilon \in (0, 1]$. As a consequence, for each $\epsilon \in (0,1]$, condition \eqref{e:existence about f}$_1$ implies
\begin{equation*}
\min_{\mu_0\le \rho \le \mu_1}\left[f(\rho)-\left(f(\alpha) + c_{\ell_{\pm}}(\rho-\alpha)\right)\right]=\sigma,
\end{equation*}
for some $\sigma >0$.
On the other hand, by \eqref{e:deriv} we have
\begin{equation*}
\phi_{\epsilon}^{\, \prime}(\xi)\ge  \frac{\sigma}{\epsilon D\left(\phi_{\epsilon}(\xi)\right)} \ge \frac{\sigma}{\epsilon \displaystyle{\max_{\mu_0\le \rho\le \mu_1}} D(\rho)}, \qquad\xi \in [\xi_0, \xi_1], \, \epsilon \in (0,1],
\end{equation*}
which contradicts that $\phi_{\epsilon}(0)=\alpha$ for every $\epsilon>0$.

At last, by the estimate
\begin{equation*}
\phi_{\epsilon}(\xi)\le \phi_{\epsilon}(\delta), \quad \xi \ge \delta, \, \, \epsilon >0,
\end{equation*}
we obtain the uniform convergence  in any half line $[\delta, \infty)$, with $\delta >0$; the reasoning is similar in any half line $(-\infty, -\delta)$, again with $\delta >0$.
\end{proofof}

\begin{proofof}{Lemma \ref{l:D1no}}
First, assume that $v$ does not decrease. If $v'$ changes sign once, then $v$ first decreases and then increases. This contradicts the assumptions $v\ge0$ in $[0,1)$ and $v(1)=0$. If $v'$ changes sign twice, then it vanishes twice in $(0,1)$ and the same must occur for $D$. This contradicts (D1).

We are left with the case when $v$ decreases. We denote $w(\rho) = \delta + \tau\rho v'(\rho)$. In order that (D1) holds we need that $w$ changes sign from the positive to the negative at $\alpha$.

If $w(\rho) =k(\alpha-\rho)$ for some $k>0$, then $v(\rho) = \frac{k\alpha-\delta}{\tau}\log\rho + \frac{k}{\tau}(1-\rho)$, with $k\alpha\le \delta$ in order that $v$ may decrease. Indeed, we need $k\alpha=\delta$ in order that $v$ is bounded and then, by imposing $v(1)=1$, we have $\eqref{e:vW}_1$.
This velocity satisfies (fcm) and we deduce $D(\rho)= \frac{\delta^2}{\alpha^2\tau}\rho(\alpha-\rho)$.

This example has the drawback that $D(1)\ne0$; if we define for instance $w(\rho) =k(\alpha-\rho)(1-\rho)$, then, choosing again $k\alpha=\delta$ and imposing $v(1)=0$, we deduce $\eqref{e:vW}_2$ with $D(\rho) = \frac{\delta^2}{\alpha^2\tau}\rho(\alpha-\rho)(1-\rho)(1+\alpha-\rho)$.

About condition \eqref{e:existence about f}, in the case $\eqref{e:vW}_1$ the function $f$ is concave. In the latter, it is easy to prove that any line through $\left(\alpha,f(\alpha)\right)$ intersects the graph of $f$ at most at one point different from $\alpha$. At last, it is obvious from \eqref{e:vW} that condition \eqref{e:vvs} fails in both cases.
\end{proofof}

\begin{proofof}{Lemma \ref{l:quadraticv}}
The line $y=\sigma \rho$ meets the graph of $(1-\rho)^3$ precisely at one point in the interval $(0,1)$; this defines $\alpha\in(0,1)$ by \eqref{e:sigma} and then $D$ satisfies (D1). Note that $\alpha$ covers the interval $(0,1)$ when $\sigma$ ranges in $(0,\infty)$; in particular we can take $\alpha\in(\frac12,1)$ for a suitable choice of the parameters $\ov,h,\tau$.

The pairs $(\ell^-,\ell^+)$ are constructed by a geometric argument. Denote by $r_m(\rho)= m(\rho-\alpha) + f(\alpha)$ a generic line through $\left(\alpha, f(\alpha)\right)$, parametrized by $m$; we need that the equation $f(\rho) = r_m(\rho)$ has two solutions $\rho_\pm$ with
\begin{equation}\label{e:rhopm}
0<\rho_-<\alpha<\rho_+<1.
\end{equation}
We have $f(\rho) = r_m(\rho)$ if and only if $\ov\rho(1-\rho)^2 = m(\rho-\alpha) + \ov\alpha(1-\alpha)^2$. Since $\rho(1-\rho)^2-\alpha(1-\alpha)^2=(\rho-\alpha)\left(\rho^2 -(2-\alpha)\rho+(1-\alpha)^2\right)$,  then the previous equation is satisfied iff
\begin{equation}\label{e:snde}
\rho^2 -(2-\alpha)\rho+(1-\alpha)^2-\mu=0,\quad \mu:=m/\ov.
\end{equation}
We need that \eqref{e:snde} has roots $\rho_\pm$ satisfying \eqref{e:rhopm}. The discriminant 
of \eqref{e:snde} is positive iff
\begin{equation}\label{e:mu1}
\mu>-\alpha\left(1-\frac34\alpha\right).
\end{equation}
If \eqref{e:mu1} holds then the roots of \eqref{e:snde} are
\(
\rho_\pm = \frac{2-\alpha\pm\sqrt{-3\alpha^2+4\alpha+4\mu}}{2}.
\)
Now, we check \eqref{e:rhopm}.

$\bullet$ We have $0<\rho_-$ iff $2-\alpha>\sqrt{-3\alpha^2+4\alpha+4\mu}$, i.e.,
\begin{equation}\label{e:mu2}
\mu<(1-\alpha)^2.
\end{equation}

$\bullet$ We have $\rho_-<\alpha$ iff $2-3\alpha<\sqrt{-3\alpha^2+4\alpha+4\mu}$, i.e.,
\begin{equation}\label{e:mu3}
\hbox{ either } \alpha\in\left(2/3,1\right)\quad  \hbox{ or } \quad \alpha\in\left(0,2/3\right] \hbox{ and } \mu>3\alpha^2-4\alpha+1=(1-\alpha)(1-3\alpha).
\end{equation}

$\bullet$ We have $\alpha<\rho_+$ iff $3\alpha-2<\sqrt{-3\alpha^2+4\alpha+4\mu}$, i.e.,
\begin{equation}\label{e:mu4}
\hbox{ either } \alpha\in\left(0,2/3\right)\quad  \hbox{ or } \quad \alpha\in\left[2/3,1\right) \hbox{ and } \mu>3\alpha^2-4\alpha+1=(1-\alpha)(1-3\alpha).
\end{equation}

$\bullet$ We have $\rho_+<1$ iff $\sqrt{-3\alpha^2+4\alpha+4\mu}<\alpha$, i.e.,
\begin{equation}\label{e:mu5}
\mu<-\alpha(1-\alpha).
\end{equation}
Therefore $\mu$ must satisfy conditions \eqref{e:mu1}, \eqref{e:mu2}, \eqref{e:mu3}, \eqref{e:mu4}, \eqref{e:mu5}. Clearly \eqref{e:mu5} implies \eqref{e:mu2}; on the other hand, \eqref{e:mu3} and \eqref{e:mu4} imply \eqref{e:mu1} because $3\alpha^2-4\alpha+1\ge-\alpha\left(1-\frac34\alpha\right)$. Then we require
\begin{equation}\label{e:finalmu}
(1-\alpha)(1-3\alpha)<\mu < -\alpha(1-\alpha).
\end{equation}
Notice that $\mu<0$ and that the inequality is nonempty iff $\alpha>\frac12$, whence the requirement in the statement. Then the pairs $(\ell^-,\ell^+)$ are found as follows: choose any $m$ in the interval $\left(\ov(1-\alpha)(1-3\alpha), -\ov\alpha(1-\alpha)\right)$; then the abscissas (different from $\alpha$) of the points of intersection of the line $r_m$ with the graph of $f$ can be taken as $\ell^\pm$.
\end{proofof}

\begin{proofof}{Lemma \ref{l:Q}}
We have $v'(\rho)=-\rho^2+(\alpha+\beta)\rho-\alpha\beta$. By integration and imposing the condition $v(1)=0$ we deduce
\[
v(\rho) = -\frac{\rho^3}{3}+ \frac{\alpha+\beta}{2}\rho^2 -\alpha\beta\rho + \gamma,\quad
\gamma:= \frac13-\frac{\alpha+\beta}{2}+\alpha\beta.
\]
We need to check when $v(\rho)\ge0$ for $\rho\in[0,1)$. By \eqref{e:v3}, in the interval $[0,1]$ the function $v$ has two minimum points at $\alpha$ and $1$. We need to prove when $v(\alpha)\ge0$. We have
\[
v(\alpha) = \frac{\alpha^3}{6}-\frac{\beta}{2}\alpha^2+\left(\beta-\frac12\right)\alpha
-\frac{\beta}{2}+\frac13=\frac16(\alpha-1)^2\left(\alpha-3\beta+2\right).
\]
At last, under \eqref{e:Dalphabeta}-\eqref{e:vv}, the velocity $v$ is decreasing in $[0,\alpha)\cup(\beta,1]$ and increasing in $(\alpha,\beta)$. Hence $v$ has a minimum at $\alpha$ and $v(\alpha)<v(\beta)$. Then
\[
\frac{f(\beta)-f(\alpha)}{\beta-\alpha} = \frac{\beta v(\beta)-\alpha v(\alpha)}{\beta-\alpha} >0.
\]
This proves the lemma.
\end{proofof}

\section*{Acknowledgments}
The authors are members of the {\em Gruppo Nazionale per l'Analisi Matematica, la Probabilit\`{a} e le loro Applicazioni} (GNAMPA) of the {\em Istituto Nazionale di Alta Matematica} (INdAM) and acknowledge financial support from this institution.

{\small
\bibliography{refe_panic}
\bibliographystyle{abbrv}
}
\end{document}